\documentclass[11pt,reqno]{amsart}

\title[Remarks on discrete subgroups with full limit sets]{Remarks on discrete subgroups with full limit sets in higher rank Lie groups}
\author{Subhadip Dey}
\address{School of Mathematics, Tata Institute of Fundamental Research, Mumbai 400005, India
}
\email{subhadip@math.tifr.res.in}

\author{Sebastian Hurtado}
\address{Department of Mathematics, Yale University, New Haven, CT 06511}
\email{sebastian.hurtado-salazar@yale.edu} 


\date{\today}

\usepackage[inline]{enumitem}
\setenumerate{label=(\roman*),itemsep=0pt}
\usepackage{amssymb}
\usepackage[dvipsnames]{xcolor}
\usepackage{mathtools}
\usepackage[T1]{fontenc}
\usepackage[utf8]{inputenc}
\usepackage{hyperref} 
\hypersetup{
    colorlinks=true,
    linkcolor=blue,
    hypertexnames=false,
    filecolor=black,      
    urlcolor=black,
    citecolor=blue,
    linktocpage=true,
    unicode=true
}

\usepackage[]{cleveref}
\theoremstyle{plain}
\newtheorem{theorem}{Theorem}[section]
\newtheorem{lemma}[theorem]{Lemma}
\newtheorem{proposition}[theorem]{Proposition}
\newtheorem{corollary}[theorem]{Corollary}
\newtheorem{question}[theorem]{Question}
\theoremstyle{definition}

\theoremstyle{definition}
\newtheorem{definition}[theorem]{Definition}

\newtheorem{condition}[theorem]{Condition}
\theoremstyle{remark}
\newtheorem{remark}[theorem]{Remark}

\newcommand{\op}[1]{\operatorname{#1}}

\def\a{\alpha}

\def\B{\mathcal{B}}
\def\C{\mathcal{C}}

\def\E{\mathcal{E}}

\def\F{\mathcal{F}}
\def\G{\Gamma}
\def\g{\gamma}

\newcommand{\Hquad}{\hspace{0.5em}}
\newcommand{\GE}{\hspace{0.5em}\ge\hspace{0.5em}}
\newcommand{\LE}{\hspace{0.5em}\le\hspace{0.5em}}

\def\CC{\mathbb{C}}

\DeclareMathOperator{\Flag}{Flag}

\DeclareMathOperator{\Hdim}{dim_{\it H}}
\DeclareMathOperator{\Pdim}{\overline{dim}_{\it p}}
\DeclareMathOperator{\PP}{\mathbb{P}}

\DeclareMathOperator{\SL}{SL}
\DeclareMathOperator{\PSL}{PSL}
\DeclareMathOperator{\SO}{SO}

\newcommand{\mf}[1]{\mathfrak{#1}}
\def\<{\langle}
\def\>{\rangle}

\def\LL{\mathcal{L}}
\def\N{\mathbb{N}}
\def\R{\mathbb{R}}

\def\Z{\mathbb{Z}}
\def\liea{\mathfrak{a}}

\DeclareMathOperator{\rchi}{{\mathpalette\irchi\relax}}
\newcommand{\irchi}[2]{\raisebox{\depth}{$#1\chi$}}

\usepackage{soul}

\begin{document}

\begin{abstract}
    We show that real semi-simple Lie groups of higher rank contain (infinitely generated) discrete subgroups with full limit sets in the corresponding Furstenberg boundaries. Additionally, we provide criteria under which discrete subgroups of $G = \SL(3,\R)$ must have a full limit set in the Furstenberg boundary of $G$.
    
    In the appendix, we show the existence of Zariski-dense discrete subgroups $\G$ of $\SL(n,\R)$, where $n\ge 3$, such that the Jordan projection of some loxodromic element $\g \in\G$ lies on the boundary of the limit cone of $\G$.
\end{abstract}

\maketitle

\setcounter{tocdepth}{1}
\tableofcontents

\section{Introduction}

Let $G$ be a real non-compact semi-simple Lie group with finite center (e.g. $\SL(n, \R)$, $ n \geq 2$), $X$ be the corresponding symmetric space where $G$ acts by isometries, $P$ be a minimal parabolic subgroup of $G$, and $G/P$ be the Furstenberg boundary of $G$.
It is known that if $\Gamma$ is a Zariski-dense discrete subgroup of $G$, then there exists a unique nonempty minimal $\Gamma$-invariant closed set $\Lambda_{\Gamma} \subset G/P$, known as the {\em limit set} of $\Gamma$. In the case where $\Gamma$ is a lattice in $G$, the limit set is {\em full}, i.e. $\Lambda_{\Gamma} = G/P$.

\medskip

Our motivation for this article is the following question:

\begin{question}\label{ques:minimal_lattice}
Suppose that $G$ is simple real algebraic group with real rank at least two (e.g. $\SL(n,\R)$, $n \geq 3$) and let $\Gamma$ be a finitely generated Zariski-dense discrete subgroup of $G$. If $\Gamma$ has a full limit set, i.e. $\Lambda_\Gamma = G/P$, then must  $\Gamma$ be a lattice in $G$?
\end{question}

It is known that \Cref{ques:minimal_lattice} has a negative answer for rank one groups (cf. \Cref{rem:full_rank_one}).
However, there are some reasons to believe why this assertion could be true in higher rank. It was a great achievement in the theory of finitely generated Kleinian groups (i.e., discrete subgroups of $\PSL(2,\CC)$) that a full limit set implies ergodicity. (This follows from the work of Thurston, Canary, Calegari-Gabai, and Agol, see \cite{canary2010marden}). Thus, one can expect that the following open question has a positive answer:

\begin{question}\label{ques:minimal_ergodic}
Let $G$ and $\Gamma$ be as in \Cref{ques:minimal_lattice}. If $\Gamma$ has a full limit set, then must $\Gamma$ act ergodically on $G/P$?
\end{question}

Moreover, in the higher rank case, the following has been conjectured by Margulis:\footnote{This is a question originally due to Schoen-Yau in the more general setting of Riemannian manifolds of non-positive curvature. See problem 47, page 387, \cite{Schoen-Yau-Lectures}}

\begin{question} [Margulis conjecture] \label{ques:margulis_conj}
Let $G$ be as in \Cref{ques:minimal_lattice}.
If a discrete subgroup $\Gamma$ of $G$ acts ergodically on $G/P$ (equivalently, $X/\Gamma$ does not admit a non-constant bounded harmonic function),
then must $\Gamma$ be a lattice in $G$?
\end{question}

Although this question remains unsolved, using Kazhdan's property (T), Margulis \cite{margulis1997existence} proved that if $\Gamma$ acts ergodically on the product of $G/P \times G/P$ (or, equivalently, $G/A$), then $\Gamma$ is a lattice. Furthermore, a recent breakthrough by Fraczyk-Gelander \cite{fraczyk2023infinite} represents a significant advancement in this direction; they showed that if $X/\Gamma$ has a uniform upper bound on its injectivity radius, then $\Gamma$ is a lattice in $G$.

\medskip

A positive answer to \Cref{ques:minimal_lattice} can also provide a reasonable test to decide when a given subgroup of a lattice is thin or a finite index subgroup. To illustrate this, we show in \Cref{sec:criterion} the following criteria in the special case when $\Gamma = \SL(3,\Z)$ and $G = \SL(3,\R)$:

\begin{theorem}\label{intro:thm1}
Let $\G$ be a Zariski-dense subgroup of $\SL(3,\R)$, which is contained in $\SL(3,\Z)$ and satisfies one of the following:
\begin{enumerate}
\item $\Gamma$ contains a subgroup isomorphic to $\Z^2$.\label{intro:thm1:case1}
\item $\Gamma$ contains a standard copy of $\SL(2,\Z)$.
\item $\Gamma$ contains an infinite order element with one complex (non-real) eigenvalue. \label{intro:thm1:case3}
\end{enumerate}
Then $\Gamma$ acts minimally on $G/P$.
\end{theorem}

Therefore, an affirmative answer \Cref{ques:minimal_lattice} would imply the following about subgroups of $\SL(3,\Z)$:
\begin{enumerate}
\item $\Z^2 \star \Z$ is not isomorphic to a subgroup of $\SL(3,\Z)$; this is a long standing question of M. Kapovich, see \cite{kapovich2023list}.
\item If $\Gamma \subset \SL(3,\Z)$ contains a copy of $\SL(2,\Z)$ and is Zariski-dense, then $\Gamma$ must be a lattice.
\end{enumerate}

\begin{remark}\label{rem:full_rank_one}
    (1) 
    The analog of \Cref{ques:minimal_lattice} for rank-one real Lie groups $G$ has a negative answer. Riley \cite{Riley} discovered that the figure-eight knot complement admits a finite-volume hyperbolic structure, making it the first known hyperbolic knot. This 3-manifold was already known to fiber over the circle \cite{BurdeZieschang,Goldsmith}. See also J\o{}rgensen \cite{Jorgensen}. Soon after, Thurston \cite{Thurston-fiber} developed a deformation theory of surface group representations into $\operatorname{PSL}(2,\mathbb{C})$, showing that many closed hyperbolic 3-manifolds fiber over the circle. The fiber subgroups are infinite-index normal subgroups of $\pi_1(M)$ with full limit set in $\partial_\infty \mathbb{H}^3 = \mathbb{S}^2$.

    Higher-dimensional examples of finitely generated Kleinian groups with full limit sets were pointed out by Bruno Martelli: The work of Kielak \cite{Kielak} shows that finitely generated RFRS groups with vanishing first $L^2$-Betti number virtually admit algebraic fibrations, i.e. homomorphisms onto $\mathbb{Z}$ with finitely generated kernel. By the work of Bergeron–Haglund–Wise \cite[Thm.~1.10]{Bergeron-Haglund-Wise}, arithmetic groups of the first type in $\SO_0(n,1)$ virtually embed in right-angled Coxeter groups and, by Agol’s result \cite{Agol-virtual-fibering}, are thus virtually RFRS. It then follows from Kielak’s theorem that such arithmetic groups virtually algebraically fiber; the kernels of such fibrations are finitely generated and have full limit set. For more explicit examples in $\SO_0(n,1)$, where $4 \leq n \leq 8$, we refer to the work of Italiano–Martelli–Migliorini \cite{italiano2020hyperbolic} and Jankiewicz–Norin–Wise \cite{JNW}.

    Stover \cite[Cor. 1.4]{Stover-cusp} exhibited the existence of finitely generated complex hyperbolic Kleinian groups with infinite covolume and full limit sets in all  dimensions $n \ge 2$.
    
    To the best of our knowledge, examples of such discrete subgroups are not known in the other rank one cases, i.e., when $G= \operatorname{Sp}(n,1)$ with $n\geq 2$ and $G =F_4^{-20}$.

    \smallskip\noindent
    (2) Among rank-one Lie groups, negative answers to \Cref{ques:minimal_ergodic} are known in at least two cases: $G = \operatorname{PU}(n,1)$ for $n = 2,3$. For $n=2$, Livné \cite{Livne}, and for $n=3$, Deraux \cite{Deraux-forgetful}, exhibited closed complex hyperbolic $n$-manifolds admitting non-constant holomorphic maps with connected fibers onto hyperbolic Riemann surfaces $\Sigma$. This induces a surjection $\varphi: \pi_1(S) \to \pi_1(\Sigma)$. Passing to the regular cover $\widetilde{M}$ of $M$ corresponding to the normal subgroup $N = \ker \varphi$, which can be taken to be finitely generated (see, e.g., \cite[Remark 3.4]{Dey-Liu}), one obtains a non-constant holomorphic map $f: \widetilde{M} \to \mathbb{D}$ onto the unit disk $\mathbb{D}$ in the complex plane, whose real part is a non-constant bounded harmonic function (see Kapovich \cite[Example 2.13]{Kapovich_survey}). Thus, $N$ does not act ergodically on $\partial_\infty \mathbb{H}_\mathbb{C}^n$, although $\Lambda_N = \partial_\infty \mathbb{H}_\mathbb{C}^n$.

\smallskip\noindent
    (3) In proving both \ref{intro:thm1:case1} and \ref{intro:thm1:case3} of \Cref{intro:thm1}, a key point is that $\SL(3,\Z)$ does not contain any singular semi-simple elements. 
    On the other hand, it is known that one can embed $\Z^2\star \Z$ discretely into $\SL(3,\R)$ and any such embedding must be Zariski-dense, see \cite{Soifer_supersingular}. However, known constructions of such discrete $\Z^2\star \Z$ in $\SL(3,\R)$  in the current literature involve a ``supersingular'' $\Z^2$, i.e., a free rank-two abelian subgroup of $\SL(3,\R)$ generated by two singular diagonalizable elements whose product is also singular.
\end{remark}

\subsection{Main result} Our main goal of this paper is to show that one cannot drop the condition of finite generation in \Cref{ques:minimal_lattice}, this answers a question raised by Lee and Oh (Question 1.5, \cite{lee2024dichotomy}).

\begin{theorem}\label{thm:main}
Every non-compact connected real semi-simple Lie group $G$ with finite center contains an infinitely generated discrete subgroup $\G$ that acts minimally on the Furstenberg boundary of $G$. 
\end{theorem}

As lattices in $ G $ are finitely generated, it follows that the group $ \Gamma $ described in \Cref{thm:main} cannot be a lattice.
Consequently, our main result exhibits existence of infinite co-volume discrete subgroups of $G$ with full limit sets in the Furstenberg boundary of $G$.

\begin{remark}
Our argument also shows that every lattice $\Delta\subset G $  contains an infinite index subgroup $\G$ with a full limit set in the Furstenberg boundary of $G$. See \Cref{mainthm_lattice}.    
\end{remark}

We remark that if $G$ is rank one, then the existence of infinite co-volume discrete subgroups with full limit sets also follow from the theory of hyperbolic groups. More precisely, if $ \Gamma \subset G $ is a (torsion-free) uniform lattice in $ G $ of rank one, then $ \Gamma $ is a nonelementary hyperbolic group. Consequently, by Ol'shanski\u{\i} \cite{Olshanskii} and Delzant \cite{Delzant}, $ \Gamma $ contains an infinite index nontrivial normal subgroup $\Delta$. Although such normal subgroups $\Delta$ are not lattices in $ G $, they share the same limit set as $ \Gamma $, providing examples of infinite co-volume discrete subgroups of $ G $ with full limit sets.
However, a similar argument is hopeless if $G$ is of higher rank due to Margulis' Normal Subgroup Theorem, and the existence of discrete subgroups of $G$ with a full limit set that are not lattices, which is the main content of \Cref{thm:main}, is most interesting when $G$ is of higher rank.

\subsection{Idea of the proof of Theorem \ref{thm:main}}

The idea of the proof of  \Cref{thm:main} is as follows: We construct inductively an infinite sequence of subgroups $\Gamma_1 \subset \Gamma_2 \subset \cdots$ of $G$ such that $\Gamma_{n+1}$ is obtained from $\Gamma_n$ by adding a new element $\beta _{n+1} \in G$ satisfying some conditions; see \S\ref{mainconstruction} for details. The final group \[ \Gamma = \bigcup_{n \in \N} \Gamma_n\]
will satisfy the conclusion of \Cref{thm:main}.

To carry out this inductive construction, we will need certain control on these subgroups: More precisely, at each step, we'll ensure that $ \Gamma_n $ is a free subgroup of rank $n$, which is also an Anosov subgroup of $G$. Additionally, we'll select the sequence $ (\beta_n) $ such that the sequence $ (x_n) $ of attractive fixed points in $ G/P $ forms a dense set. This will ensure that $ \Gamma $ has a full limit set.
In the case where $G$ has rank one, this could be achieved easily using the classical Maskit Combination Theorem. However, in the higher rank case, this argument requires special care, as one needs $x_{n+1}$ and $\Lambda_{\Gamma_{n}}$ to be in general position with respect to each other; even if $\Lambda_{\Gamma_{n}}$ is not full, this might not be possible a priori.\footnote{In this regard, a negative answer to the following question would have simplified some arguments: Can there be an Anosov subgroup $\Delta$, isomorphic to a free group, of a simple real Lie group $G$ of rank at least two, such that the set of points in $G/P$ not in general position with respect to $\Lambda_\Delta$ has a nonempty interior?}
    
In order to circumvent this issue, we will guarantee that for each $n\in\N$, the Hausdorff dimension of the limit $\Lambda_{\Gamma_n}$ is sufficiently small (always strictly bounded by one). This will ensure the existence of $x_{n+1}$ such that the final sequence $(x_n)$ is dense in $G/P$. In order to control the Hausdorff dimension of $\Lambda_{\Gamma_n}$, we need bounds of the Hausdorff dimension in terms of certain critical exponents that follow from 
\cite{Dey:2024aa}.  The most intricate part of our argument lies in ensuring that at each step, these critical exponents increase only by a small amount if $\beta_n$ is appropriately chosen; refer to \Cref{prop:crit} for a precise statement.

\subsection*{Acknowledgements} 
We would like to thank Sami Douba, Misha Kapovich, Bruno Martelli, Hee Oh, and Pierre Py for various remarks and interesting conversations related to this work. Special thanks to Or Landersberg and Giuseppe Martone for various discussions and computations related to the example discussed in Proposition \ref{prop:KLLR}.  We are grateful to our annonymous referee for their careful reading, corrections, and additional comments, which improved significatly this paper, especially for suggesting including Proposition \ref{samicriterion} and correcting the proof of \ref{sl3full:case1}. S.~D. completed this work at Yale University and extends his gratitude to the mathematics department for providing a wonderful working environment. S.~H. was supported by the Sloan Fellowship Foundation.

\section{Preliminaries}

The goal of this section is to introduce some definitions, set up our notation, and prove some preliminary results essential for our inductive construction provided in \S\ref{mainconstruction}.

\subsection{Basic notions}\label{sec:notation}
 Let $G$ be a non-compact connected real semi-simple Lie group with finite center, and let $\mathfrak{g}$ denote the Lie algebra of $G$. For a Cartan decomposition
$\mathfrak{g} = \mathfrak{k} \oplus \mathfrak{p}$, let $\mathfrak{a}$ be a maximal abelian subalgebra of $\mathfrak{p}$, and let $A = \exp \mathfrak{a}$ be a maximal real split torus.
The (real) rank of $G$ is the dimension of $\mathfrak{a}$.
(We refer to standard books such as Knapp \cite{Knapp-book} for more details.)

Let $\Sigma$ be the set of restricted roots of $(\mathfrak{g},\mathfrak{a})$, which comprises all nonzero elements $\beta\in\mathfrak{a}^*$ for which the associated weight space
$\mathfrak{g}^\beta = \{ X\in \mathfrak{g} :\ [H, X] = \beta(H)X, \, \forall H\in\mathfrak{a} \}$ is nonzero.
We choose a set of simple roots, denoted by $\Pi\subset \Sigma$, which is a basis for the dual space $\mathfrak{a}^*$ such that every $\beta\in\Sigma$ can be expressed as a positive or negative integer linear combination of elements of $\Pi$.
Let $\Sigma_+$ (resp. $\Sigma_-$) denote the collection of positive (resp. negative) roots; $\Sigma = \Sigma_+ \sqcup \Sigma_-$. The positive Weyl chamber $\{H\in \mathfrak{a} :\ \beta(H) \ge 0,\, \forall \beta\in\Sigma^+\}$ is denoted by $\mathfrak{a}^+$.
With these choices, the Cartan decomposition of $G$ is given by $G = K A^+ K$, where $A^+ = \exp \mathfrak{a}^+$ and $K$ is a maximal compact subgroup of $G$.
The Cartan projection $\mu: G \to \mathfrak{a}^+$ assigns to each $g \in G$ the unique element $\mu(g) \in \mathfrak{a}^+$ such that $g = k_1 \exp(\mu(g)) k_2$, where $k_1, k_2 \in K$.

The Weyl group $W$ is defined as $N_K(\mathfrak{a})/Z_K(\mathfrak{a})$, which can be naturally identified with the finite group generated by reflections in the roots. There exists a unique element of $W$ that maps $\mathfrak{a}^+$ to $-\mathfrak{a}^+$. We fix an element $w_0\in N_K(\mathfrak{a})$ of order two, representing the longest Weyl element.

The normalizer of the nilpotent subalgebra $\mathfrak{n} = \bigoplus_{\beta\in\Sigma_+} \mathfrak{g}^\beta$ in $G$ is the standard minimal parabolic subgroup, denoted by $P$. The standard opposite minimal parabolic subgroup $P^- = w_0 P w_0$ is the normalizer in $G$ of $\mathfrak{n}^- = \bigoplus_{\beta\in\Sigma_-} \mathfrak{g}^\beta$.
Similarly, each simple root $\alpha\in\Pi$ defines a standard (resp. opposite) maximal parabolic subgroup $P_\a$ (resp. $P_\a^-$) of $G$, which is the normalizer of the nilpotent subalgebra $ \bigoplus_{\beta\in \Sigma^+\smallsetminus \op{span}(\Pi\smallsetminus\{\a\})} \mathfrak{g}^\beta$ (resp. $\bigoplus_{\beta\in \Sigma^-\smallsetminus \op{span}(\Pi\smallsetminus\{\a\})} \mathfrak{g}^{\beta}$) in $G$.
Note that $P\subset P_\a$ and $P^-\subset P^{-}_\a$ for all $\a\in\Pi$.

\subsection{Schubert varieties}
The quotient space $\mathcal{F} = G/P$ is the Furstenberg boundary of $G$.
 Being a $G$-homogeneous space, the Furstenberg boundary $\F$ carries a unique structure of an analytic manifold;
 see  \cite[Ch. III, Sect. 1]{Bourbaki1}.
 Given $x = g[P] \in\F$, let
 \[
  \C_x = g P w_0 [P],
 \]
 where $w_0$ is the longest Weyl group element.
 This is a dense open set in $\F$.
 For the rest of this paper, we fix a Riemannian metric on $\F$ and let $d_\F: \F\times \F \to \R$ be the associated distance function.
 For convenience, we will assume that the diameter of $\F$ is $1$.

 The complement of $\C_x$ in $\F$, denoted by $\E_x$, 
 is an analytic subvariety\footnote{If $M$ is a (real) analytic manifold and $Z\subset M$ is a closed subset, then $Z$ is called an {\em analytic subvariety} of $M$ if for all $z\in Z$, there exists an open neighborhood $U$ of $z$ and finitely many analytic functions $f_1,\dots,f_k:U\to \R$ such that $Z\cap U = \{ x\in U:\ f_k(x) = 0, \forall i=1,\dots,k\}$.}
 of $\F$, which is the finite union of {\em Schubert cells} of co-dimension at least one:
 \[
  \E_x = \bigsqcup_{w\in W, w\ne w_0} g P \tilde w [P]
 \]
 In the above, $\tilde w$ is a (equivalently, any) representative of $w\in W$ in $N_K(\mathfrak{a})$.
 Therefore, 
 \begin{equation}\label{eqn:dimE}
      \Hdim \E_x =
  \Pdim \E_x \le \dim \F - 1,
 \end{equation}
 where $\Hdim$ and $\Pdim$ denote the Hausdorff dimension and upper packing dimension in $(\F,d_\F)$, respectively.
 
More generally, for a nonempty subset $S\subset \F$, let 
\[
 \E(S) \coloneqq \bigcup_{x\in S} \E_x
 \quad\text{and}\quad
 \C(S) \coloneqq \E(S)^c = \bigcap_{x\in S} \C_x.
\]
In what follows, it will be necessary to control the size of $\E(S)$. We provide an upper bound for the size of this set as follows:

\begin{lemma}\label{lem:cs1}
 For all nonempty Lebesgue measurable subsets $S\subset \F$, we have that  $\Hdim \E(S) \le \Hdim S + \Hdim \E_{x}$.
\end{lemma}

\begin{proof}
    Since $\F$ is finitely covered by open subsets of the form $\C_y$, $y\in \F$,  it will be enough to assume that there exists a point $y_0\in \F$ such that $S\subset \C_{y_0}$.
    Fixing a base-point $x_0\in \C_{y_0}$, we can parametrize $\C_{y_0}$ by $\varphi : N_0 \to \C_{y_0}$,
    $\varphi(h) = hx_0$, where $N_0$ is the maximal unipotent subgroup  (i.e., a conjugate of $N = \exp \mathfrak{n}$) of $G$ stabilizing $y_0$.
    The map $\varphi$ is a diffeomorphism.
    Let us fix a Riemannian metric on $N_0$.
    As $\varphi$ is a diffeomorphism, for $Z \coloneqq \varphi^{-1} (S) \subset N_0$,
    we have $\Hdim Z = \Hdim S$.

    Let us consider the smooth map
    $
     \psi : N_0 \times \F \to \F, 
    $
    $
     \psi (h,x) = hx.
    $
    Since $\psi$ is smooth, it is locally Lipschitz. Since Hausdorff dimension is nonincreasing under Lipschitz maps,
    we have $\Hdim\psi (Z\times \E_{x_0}) \le \Hdim (Z\times \E_{x_0})$. 
    Moreover, by Tricot \cite[Sect. 5]{Tricot} (see also \cite[Thm. 8.10]{Mattila}), $\Hdim (Z \times \E_{x_0})$ is bounded above by $\Hdim Z + \Pdim \E_{x_0}$.
    Thus, by \eqref{eqn:dimE}, we have
    \begin{equation}\label{eqn:lem:cs1}
        \Hdim\psi (Z\times \E_{x_0})\le \Hdim Z + \Hdim \E_{x_0}.
    \end{equation}
    
    Finally, observe that
    $\psi (Z\times \E_{x_0}) = \bigcup_{h\in Z} h\E_{x_0} = \bigcup_{h\in Z} \E_{h x_0} = \E(S)$.
    So, by \eqref{eqn:lem:cs1}, we get $\Hdim \E(S) \le \Hdim S + \Hdim \E_{x_0}$.
\end{proof}

The following result is a direct consequence of \eqref{eqn:dimE} and \Cref{lem:cs1}:

\begin{corollary}\label{lem:cs}
 Suppose that $S$ is a nonempty measurable subset in $\F$.
 If $\Hdim (S) <1$,
 then $\Hdim \E(S) < \dim \F$. 
 Consequently, in this case, $\op{Leb} \E(S) = 0$ and $\C(S)$ is full measure subset $\F$.
\end{corollary}

\subsection{Anosov subgroups}
The critical exponent is an important numerical invariant associated with discrete groups.
Given a discrete subgroup $\G$ of $G$, the {\em critical exponent} of $\G$ with respect to a simple root $\a\in\Pi$ is defined by
\begin{equation}\label{eqn:crit}
 \delta_\a(\G) \coloneqq \limsup_{R\to\infty} \frac{\log \# \{ \g\in\G :\ \a(\mu(\g)) \le R\}}{R}.
\end{equation}
While $\delta_\a(\G)$ takes values in the interval $[0,\infty]$, for the special class of discrete subgroups of $G$ defined below, this value is finite:

\begin{definition}\label{def:Anosov}
    A finitely generated subgroup $\G$ of $G$ is called {\em Anosov} if there exist constants $L\ge 1$ and $A\ge 0$ such that for all $\a\in\Pi$ and $\g\in\G$, 
    \[
     \a(\mu(\g)) \ge \frac{1}{L} |\g| - A,
    \]
    where $|\cdot|$ denotes some word metric on $\G$.
\end{definition}

We refer the reader to \cite{KLP_survey,WienhardICM} for surveys on the topic of Anosov subgroups.

If $\G$ is an infinite Anosov subgroup of $G$, then there exists a certain nonempty closed $\G$-invariant subset of $\F$, called the {\em limit set} of $\G$, denoted by $\Lambda_{\G}$.
Moreover, if $\G$ is non-elementary (equivalently, not virtually cyclic), then $\G$ acts minimally on $\Lambda_{\G}$.
Furthermore, any two distinct points $x^\pm\in\Lambda_\G$ are in {\em general position} (or {\em antipodal}), meaning that $x^+\in \C_{x^-}$ or, equivalently, $x^-\in \C_{x^+}$.

In the sequel, we will need an upper bound on the Hausdorff dimension of the limit set in terms of the critical exponents \eqref{eqn:crit}:

\begin{theorem}[Dey-Kim-Oh, {\cite[Thm. 1.7]{Dey:2024aa}}]\label{eqn:hdimbound}
    If $\G$ is an Anosov subgroup of $G$, then
\begin{equation*}
 \Hdim (\Lambda_{\G}) \LE \max_{\a\in\Pi}\delta_{\a}.
\end{equation*}
\end{theorem}

\subsubsection{Contraction dynamics}

\begin{lemma}\label{lem:contraction}
    Let $\G$ be an Anosov subgroup of $G$ and let $B\subset \C(\Lambda_\G)$ be a nonempty compact subset.
    Then, for all sequences $(\g_i)$ in $\G$ of pairwise distinct elements, 
    \[
     \lim_{i\to\infty}\op{diam} \g_i B = 0
     \quad\text{and}\quad
     \lim_{i\to\infty} d_\F(\g_i B, \Lambda_\G) = 0.
    \]
    In particular, for all $y\in \C(\Lambda_\G)$, the closure of the $\G$-orbit of $y$ in $\F$ is $(\G\cdot y) \sqcup \Lambda_{\G}$.
\end{lemma}

\begin{proof}
    We use the fact that Anosov subgroups are {\em regular} in the sense of \cite[Sect. 4]{KLP_char}.
    Since $B\subset \C(\Lambda_{\G})$, together with \cite[Lem. 3.8]{dey2023kleinmaskit}, the regularity of the sequence $(\g_i)$ in $G$ implies $\lim_{i\to\infty}\op{diam} \g_i B = 0$.

    Now we prove $\lim_{i\to\infty} d_\F(\g_i B, \Lambda_{\G}) = 0$. Suppose, to the contrary, that there exists a sequence $(\g_i)$ in $\G$ of pairwise distinct elements such that $d_\F(\g_i B, \Lambda_{\G}) \ge \epsilon>0$ for all $i\in\N$. Since $\F$ is compact and $\lim_{i\to\infty}\op{diam} \g_i B = 0$, after extraction, there exists $y\in\F$ such that 
    \begin{equation}\label{eqn:lem:contraction}
     y\not\in\Lambda_{\G} \quad\text{and}\quad \g_i B \to \{y\} \quad\text{as }i\to\infty.
    \end{equation}
    Moreover, since $(\g_i)$ is a regular sequence in $G$, after further extraction, there exists $x^\pm\in\Lambda_{\G}$ such that the sequence of maps
    \[
     \g_i\vert_{\C_{x^-}} : \C_{x^-} \to \F
    \]
    converges uniformly on compacts as $i\to\infty$ to the constant map $f : \C_{x^-} \to \F$, $f\equiv x^+$; see \cite[Sect. 4]{KLP_char}.
    As $B\subset \C_{x^-}$, we have $\g_i B \to \{x^+\} \subset \Lambda_{\G}$, contradicting \eqref{eqn:lem:contraction}.
\end{proof}

Let $G = KA^+K$ be a Cartan decomposition of $G$.
An element $a\in A^+$ is \textit{regular} if $a = \exp(X)$ for some $X\in \mathfrak{a}^+$ satisfying $\a(X)\ne 0$ for all $\a\in\Sigma$.
An element of $A^+$ is {\em singular} if it is not regular.

\begin{lemma}\label{lem:k}
    Let $\G$ be an Anosov subgroup of $G$. For each $\g\in\G$, choose a decomposition $\g = k_1(\g) a(\g) k_2(\g)$ such that $k_i(\g)\in K$ and $a(\g)\in A^+$.
   Then, for all nonempty compact sets $B \subset \C(\Lambda_{\G})$,
   \[
     \liminf\{ d_\F (k_2(\g) B,\, \E_{x^-}):\ \g\in\G\} \Hquad > \Hquad 0,
   \]
   where $x^-$ is the repulsive fixed point in $\F$ of a (equivalently, any) regular element $a\in A^+$.
\end{lemma}

\begin{proof}
    Let $B \subset \C(\Lambda_{\G})$ be any nonempty compact subset.
Suppose, contrary to our claim, that there exists a sequence $\g_i \in\G$ of pairwise distinct elements such that
\begin{equation}\label{eqn:lem:k}
	d_\F (k_2(\g_i) B,\, \E_{x^-}) \to 0 
	\quad \text{as } i \to \infty.
\end{equation}
For small enough $\epsilon>0$, the closed $\epsilon$-neighborhood $B'$ of $B$ lies in $\C(\Lambda_{\G})$. 
Since the action of $K$ on $(\F,d_\F)$ is uniformly bi-Lipschitz, \eqref{eqn:lem:k} implies that $k_2(\g_i) B' \cap \E_{x^-} \ne \emptyset$ for all large enough $i \in\N$.

Moreover, since $a\E_{x^-} = \E_{ax^-} = \E_{x^-}$ for all $a\in A$, we deduce that 
$
(a(\g_i)k_2(\g_i) B') \cap \E_{x^-} = a(\g_i)(k_2(\g_i) B'\cap \E_{x^-}) \ne \emptyset
$
for all large enough $i\in\N$. 
Since $\E_{x^-}$ is compact, after extraction, we obtain 
\begin{equation}\label{eqn2:lem:k}
d_\F(a(\g_i)k_2(\g_i) B',\, z) \to 0 \quad \text{as } i \to \infty
\end{equation}
for some $z\in\E_{x^-}$.
However, since $B'$ is a compact subset of $\C(\Lambda_{\G})$, by \Cref{lem:contraction}, $\operatorname{diam} \g_i B' \to 0$ as $i \to \infty$.
Therefore, we have that $\operatorname{diam} (a(\g_i) k_2(\g_i) B') \to 0$. After further extraction, there exists a compact set $B''\subset \C_{x^-}$ with a nonempty interior such that $k_2(\g_i) B \supset B''$ for all large enough $i$. Consequently, $\operatorname{diam} a(\g_i) B'' \to 0$ as $i \to \infty$. Therefore,
by \eqref{eqn2:lem:k}, 
$a(\g_i) B'' \to \{z\}$
as $i \to \infty$.

Let $x^+$ denote the attractive fixed point in $\F$ of a (equivalently, any) regular element $a\in A^+$.
 Since  $B''\cap \E_{x^-} = \emptyset$, we also obtain that $a(\g_i) B'' \to x^+$ as $i \to \infty.$
 So, by the preceding paragraph, $z=x^+$.
 However, $x^+$ and $x^-$ are in general position, i.e., $x^+\not\in \E_{x^-}$. This is a contradiction with the fact that $z\in\E_{x^-}$.
\end{proof}

\subsection{Certain bad sets are null}\label{sec:bad_g_points}

For $g\in G$, let
 $
 \B_g = \{ (x,y)\in \F^2 :\ 
  g x \in\E_x \cup \E_y \text{ or }
  g y \in\E_x \cup \E_y\}
 $
and for $\G\subset G$, let 
\begin{equation}\label{def:badsetforGamma}
      \B(\G) \coloneqq \bigcup_{\g\in\G,\, \g\ne 1_G}\B_\g
\end{equation}

At the $n$-th step of the inductive construction that we discuss in \S\ref{mainconstruction}, we need to find a two points in the Furstenberg boundary $\F$ that are in general position relative to each other and to the limit set of the group $\G_{n-1}$ constructed in the $(n-1)$-th step. It turns out that to carry out the construction of $\G_{n}$, it will also be necessary to pick the pair outside of $\B(\G_{n-1})$.
Thus, we need to show that $\B(\G_{n-1})$ is negligible. More precisely, we will prove the following:

\begin{proposition}\label{prop:badsetnull}
 If $\G$ is a  non-elementary torsion-free Anosov subgroup of $G$,
 then $\B(\G)$ has a zero Lebesgue measure in $\F^2$.
\end{proposition}

As a consequence, we obtain:

\begin{corollary}\label{cor:badsetnull}
 If $\G$ is a torsion-free Anosov subgroup of $G$,
 then $\B(\G)\cap \C(\Lambda_\G)^2$ is a closed subset of $\C(\Lambda_\G)^2$. Moreover, if $\G$ is non-elementary, then $\B(\G)\cap \C(\Lambda_\G)^2$ is nowhere dense in $\C(\Lambda_\G)^2$.
\end{corollary}

\begin{proof}
We first prove that $\C(\Lambda_\G)^2 \cap \B(\G)$ closed in $\C(\Lambda_\G)^2$.
Let $(x_n,y_n)\in\C(\Lambda_\G)^2 \cap \B(\G)$ be a sequence converging to $(x,y)\in \C(\Lambda_\G)^2$. We will show that $(x,y) \in \B(\G)$.

By definition, there exists a sequence $\g_n\in\G\smallsetminus\{1_G\}$ such that $(x_n,y_n)\in \B_{\g_n}$.
We first claim that the entries of such a sequence $(\g_n)$ form a finite subset of $\G$: If not, after extraction, we can assume that $(\g_n)$ has no repeated entries. Choose disjoint open neighborhoods $D$ and $U$ of  $\{x,y\}$ and  $\E(\Lambda_\G)$, respectively.
Consider the compact set  $Z = \{x,y\} \cup \{ x_n ,y_n:\ n\in\N\} \subset \C(\Lambda_\G)$.
By Lemma \ref{lem:contraction}, $\operatorname{diam}(\g_n^{-1} Z) \to 0$ and $d_\F(\g_n^{-1} Z,\Lambda_\G) \to 0$ as $n\to\infty$.
Thus, it follows that for all large $n$, $\g_n^{-1} \E(Z) = \E(\g_n^{-1} Z) \subset U$ and, so, $\g_nD\cap \E(Z)= \g_n(D\cap \g_n^{-1} \E(Z)) = \emptyset$. However, for all large $n$, $\g_n\{x_n,y_n\} \subset \g_n D$  and, so, $\g_n\{x_n,y_n\}\cap (\E_{x_n} \cup \E_{y_n}) = \emptyset$, contradicting $(x_n,y_n)\in \B_{\g_n}$.

Since $\B_g$ is closed in $\F^2$ for all $g\in G$, by the conclusion of the preceding paragraph, we must have that $(x,y) \in \B_{\g_n}$ for some $n$. This proves $\C(\Lambda_\G)^2 \cap \B(\G)$ closed in $\C(\Lambda_\G)^2$.

The ``moreover'' part of the result now follows from \Cref{prop:badsetnull}.
\end{proof}

The key ingredient in our proof of \Cref{prop:badsetnull}, proved at the end of this subsection, is the following lemma:

\begin{lemma}\label{lem:bad_g_points}
 If $\B_g$, $g\in G$, is a proper subset of $\F^2$, then $\B_g$ has Lebesgue measure zero.
\end{lemma}

This result will be readily derived from the following lemma:

 \begin{lemma}\label{lem:Leb_analytic}
     Let $M$ be a connected real analytic manifold and $Y\subset M$ be a closed analytic subvariety.
     Then either $Y = M$ or $\op{Leb} Y = 0$. 
 \end{lemma}

 \begin{proof}
Suppose that $Y$ has a positive Lebesgue measure in $M$. We will show that $Y=M$.
Consider the set $Z$ of all points $y\in Y$ which satisfy the following: For all $\delta>0$, $\op{Leb} (B_\delta(y)\cap Y)>0$.
Here $B_\delta(y)$ denotes the closed ball of radius $\delta$ centered at $y$.

Recall that a point $y\in Y$ is called a density point if 
\[
\lim_{r\to0} \frac{\op{Leb} (B_r(y)\cap Y)}{\op{Leb} B_r(y)} = 1.
\]
The Lebesgue density theorem (see \cite[Cor. 2.14(1)]{Mattila}) asserts that almost all points of $Y$ are density points. 
Note that $Z$ contains all the density points of $Y$.
Therefore, $\op{Leb} Z = \op{Leb} Y >0$.
In particular,
$Z$ is nonempty.

Moreover, $Z$ is an open subset of $M$, which can be seen as follows: Since $Y$ is analytic,  for all $z\in Y$ there exists a connected open neighborhood $U$ of $z$ in $M$ and analytic functions $f_1,\dots,f_k : U\to \R$
such that $Y \cap U = \bigcap_i f^{-1}_i(0)$.
If $z\in Z$, then $\op{Leb} f^{-1}_i(0) \ge \op{Leb} (Y \cap U) >0$ for each $i=1,\dots,k$. 
Since $f_i$ is analytic, it follows that $f_i\equiv 0$ on $U$ (see \cite{mityagin}) for all $i=1,\dots,k$.
Thus, $U\subset Y$, showing that $U\subset Z$.

Furthermore, $Z$ is a closed subset of $M$: If $z_i\in Z$ is a sequence converging to $y\in Y$, then for all $\delta>0$, there exists $i\in\N$ such that $B_{\delta/2}(z_i)\subset B_{\delta}(y)$, showing that $\op{Leb} (B_{\delta}(y) \cap Y) \ge \op{Leb} (B_{\delta/2}(z_i) \cap Y) > 0$. Therefore, $y\in Z$.

Thus, $Z$ is a nonempty connected component of $M$.
Since $M$ is connected, we conclude that $Z = M$. As $Z\subset Y\subset M$, it follows that $Y = M$.
 \end{proof}

We return to the proof of \Cref{lem:bad_g_points}.

\begin{proof}[Proof of \Cref{lem:bad_g_points}]
By \Cref{lem:Leb_analytic}, it is enough to show that $\B_g$ is a closed analytic subvariety of $\F^2$:
 Let $G$ act diagonally on $\F^2$ and let us consider the analytic map $\varphi_g:\F \to \F^2$ given by $\varphi_g(x)= (x,g x)$.
The subset $\{x \in\F:\ g x\in \E_x\}$, which is the $\varphi_g$-preimage of the closed analytic subvariety $G\cdot (\{x_0\}\times\E_{x_0}) \subset \F^2$, is analytic.
Thus, the set $\B^1_g$ of all points $\{(x,y)\in\F^2:\ g x\in \E_x \}$ is also an analytic subvariety of $\F^2$.
Moreover, the set $\B^2_g$ of all points $\{(x,y)\in\F^2:\ g x\in \E_y \}$ is analytic since it is the preimage of the analytic subvariety $G\cdot (\E_{y_0}\times \{{y_0}\})$ under the analytic map $\F^2 \to \F^2$, $(x,y)\mapsto (g x,y)$.
Similarly, $\B_g^3 \coloneqq \{ (x,y):\ g y\in\E_y\}$ and $\B_g^4 \coloneqq \{ (x,y):\ g y\in\E_x\}$ are closed analytic subsets of $\F^2$.
Thus, $\B_g = \bigcup_{i=1}^4 \B^i_g$ is a closed analytic subvariety of $\F^2$.
\end{proof}

Now we are ready to give:

\begin{proof}[Proof of \Cref{prop:badsetnull}]
Since $\G$ is a countable group and a countable union of measure zero sets in $\F^2$ is measure zero, it is enough to know 
that $\op{Leb}\B_\g = 0$ for all $\g\in\G\smallsetminus\{1_G\}$.
Thus, by \Cref{lem:bad_g_points}, it is enough to show that $\B_\g\subset \F^2$ is proper:
Pick any point $x \in\Lambda_{\G}$ different from $\g^\pm$, where $\g^\pm$ denote the fixed points of $\g$ in $\Lambda_{\G}$, and set $y = \g^2 x$. Then, $\g y = \g^3 x\ne x$, $\g y \ne y$, $\g x \ne x$, and $\g x =\g^{-1} y \ne y$. Since distinct points in $\Lambda_{\G}$ are pairwise in general position, it follows that $\g x\not\in \E_x \cup \E_y$ and $\g y \not\in \E_x \cup \E_y$.
 So, $(x,y)$ lies in the complement of $\B_\g$. 
\end{proof}

\section{Main construction}\label{mainconstruction}

Once and for all, we fix a countable dense subset 
$
\{ z_n :\ n\in\N\} 
$
in the Furstenberg boundary $\F$ of $G$.
Then, we will inductively construct an increasing sequence of Anosov subgroups of $G$,
\begin{equation*}
    \{1_G\} = \G_0 \subset \G_1\subset \G_2\subset\cdots
\end{equation*}
such that the following conditions hold:

\begin{condition}\label{cdn}
For all $n\in\N$,
\begin{enumerate}
 \item $\G_{n}$ is an Anosov subgroup of $G$.
 \label{cdn:1}
 
 \item for each $\a\in\Pi$, 
$
 \delta_{\a}(\G_{n}) \le 1-\frac{1}{2^{n}}.
$ \label{cdn:2}

 \item $\G_n= \<\G_{n-1},\beta_n\>$ for some element $\beta_n$ which can be conjugated to a regular element in $A^+$.
 Moreover, the natural homomorphism from the abstract free product $\G_{n-1}\star \<\beta_n\>$ to $\G_n$ is injective.
 \label{cdn:3}
 
 \item if $x_n\in\F$ denotes the attractive fixed point of $\beta_n$, then $d_\F(x_n,z_n) \le 1/n$.
 \label{cdn:4}
\end{enumerate}
\end{condition}

To initiate the construction, let $\beta_1$ be a {regular element} in $A^+$. Since $\G_1 = \<\beta_1\>$ is cyclic, one can easily verify that $\G_1$ is an Anosov subgroup (see \Cref{def:Anosov}).
The limit set of $\Gamma_1$ comprises two points in $\mathcal{F}$, specifically, the unique attractive and repulsive fixed points of $\beta_1$. Given that $\Gamma_1$ is a cyclic group, one can check easily that $\delta_\alpha(\Gamma_1) = 0$ for all $\alpha\in\Pi$, thereby satisfying condition \ref{cdn:2}. The immediate consequence of $\Gamma_0$ being the trivial subgroup of $G$ is that $\Gamma_1$ also satisfies condition \ref{cdn:3}. Additionally, since the diameter of $\mathcal{F}$ is $1$, condition \ref{cdn:4} is evidently met by $\Gamma_1$. Finally, replacing $\beta_1$ by a large enough power of it, we may also assume that $\Gamma_1$ satisfies 
\begin{equation}\label{eqn:badsetforG_1}
    \op{Leb} \B({\G_1}) = 0
\end{equation}
where $\B({\G_1})$ is as in \eqref{def:badsetforGamma}.

The remainder of this section is now devoted to proving the inductive step. Throughout this section, we assume the existence of a group $\G_{n-1}$, $n\ge 2$, satisfying the inductive hypothesis outlined in \Cref{cdn}. In \Cref{sec:pf_induction}, we will complete the construction of $\G_n$, which will also satisfy \Cref{cdn}.

\subsection{Finding a suitable pair of points}\label{sec:xpm}

Note that \Cref{cdn}\ref{cdn:2} implies that
\begin{equation*}
 \delta_{\a}(\G_{n-1}) \le 1-\frac{1}{2^{n-1}},
 \quad
 \text{for all }\a\in\Pi.
\end{equation*}
Therefore, by \Cref{eqn:hdimbound},
$
 \Hdim (\Lambda_{\G_{n-1}}) \le 1-\frac{1}{2^{n-1}}.
$
So, \Cref{lem:cs} directly implies the following:

\begin{lemma}\label{cor:CL}
 $\C(\Lambda_{\G_{n-1}})$ is an open set of full measure in $\F$.
\end{lemma}

In \S\ref{sec:ping-pong}, we will construct an Anosov subgroup of $G$, isomorphic to $\G_{n-1}\star \Z$, by adjoining a new loxodromic element $\beta \in G$ to $\G_{n-1}$. This subgroup will be constructed by a ping-pong argument. For this argument to work, we require the attractive/repulsive fixed points $x^\pm$ of $\beta$ to be in a general position relative to the limit set of $\G_{n-1}$. In addition, the pair $(x^+,x^-) \in \F^2$ must lie outside the set $\B(\G_{n-1})$ defined in \eqref{def:badsetforGamma}. The following result asserts that almost all points $(x^+,x^-)\in \F^2$ satisfy these conditions:

\begin{lemma}\label{lem:xpm}
    The set of all points $(x^+,x^-)$ in $\F^2$ satisfying the following two conditions is a dense open set:
    \begin{enumerate}
     \item $x^\pm \in \C(\Lambda_{\G_{n-1}})$, and $x^+,~x^-$ are in general position.\label{lem:xpm:cdn1}
     \item Neither $\g x^+$ nor $\g x^-$ lie in $\E(x^+)\cup\E(x^-)$, for all nontrivial elements $\g\in\G_{n-1}$.\label{lem:xpm:cdn2}
    \end{enumerate}
\end{lemma}

\begin{proof}
 \Cref{cor:CL} asserts that the open set $\C(\Lambda_{\G_{n-1}})$ has a full measure in $\F$. Consequently, the set $U$ of all points in $\F^2$ satisfying condition \ref{lem:xpm:cdn1}
 is open and dense in $\F^2$.
 
 On the other hand, we also have that $V = \C(\Lambda_{\G_{n-1}})^{2}\smallsetminus \B(\G_{n-1})$ is a dense open set in $\C(\Lambda_{\G_{n-1}})^{2}$ and, hence, in $\F^2$ (as $\C(\Lambda_{\G_{n-1}})$ is dense in $\F$): If $n\ge 3$, this conclusion arises from \Cref{cor:badsetnull} as in this case $\G_{n-1}$ is an Anosov subgroup isomorphic to a free group of rank $\ge 2$ (thus $\G_{n-1}$ is torsion-free and non-elementary).
 If $n=2$, then the conclusion follows from the first part of \Cref{cor:badsetnull} and \eqref{eqn:badsetforG_1}.
 
 Thus, the set of all points $(x^+,x^-)$ in $\F^2$ satisfying conditions \ref{lem:xpm:cdn1} and \ref{lem:xpm:cdn2}, which is the intersection of $U$ and $V$, is dense and open in $\F^2$.
\end{proof}

In the rest of this section, we fix a pair of points $x^\pm\in\F$, which satisfies conditions \ref{lem:xpm:cdn1} and \ref{lem:xpm:cdn2} of \Cref{lem:xpm} and, further, the following:
\begin{equation}\label{eqn:xpm}
    d_{F}(x^\pm,z_n) \le \frac{1}{n},
\end{equation}
where $z_n$ is as described by the first paragraph of this section.
Note that \Cref{lem:xpm} guarantees the existence of such a pair of points.
Moreover, for convenience, we fix a Cartan decomposition $G = KA^+K$ such that the attractive/repulsive fixed points in $\F$ of a (equivalently, any) regular element $a\in A^+$ is $x^\pm$, respectively.

\subsubsection{Orbits of ${x^\pm}$}
\begin{lemma}\label{cor:precpt}
Let $x^\pm\in\F$ be as above. Then:
\begin{enumerate}
    \item  $\{\g x^+ :\ \g\in \G_{n-1}\}$ and $\{\g x^- :\ \g\in \G_{n-1},\g\ne 1\}$ are precompact subsets of $\C_{x^-}$ whose closures are obtained by attaching the limit set $\Lambda_{\G_{n-1}}$ of $\G_{n-1}$.\label{cor:precpt:i}

    \item $\{\g x^- :\ \g\in \G_{n-1}\}$ and $\{\g x^+ :\ \g\in \G_{n-1}, \g\ne 1\}$  are precompact subsets of $\C_{x^+}$ whose closure is obtained by attaching $\Lambda_{\G_{n-1}}$.\label{cor:precpt:ii}

    \item there exists $\epsilon>0$ such that 
    $d_\F(k_2(\g)\{x^+,x^-\},\, \E_{x_-})\ge\epsilon$ for all but finitely many elements $\g\in \G_{n-1}$.\label{cor:precpt:iii}
\end{enumerate}
\end{lemma}

\begin{proof}
    By condition \ref{lem:xpm:cdn1} of \Cref{lem:xpm}, we have that $\Lambda_{\G_{n-1}} \subset \C_{x^-}\cap \C_{x^+}$.
    Moreover, by condition \ref{lem:xpm:cdn2} of \Cref{lem:xpm} and the fact that $x^+\not\in\E(x^-)$, we obtain that $\{\g x^+ :\ \g\in \G_{n-1}\}$ and $\{\g x^- :\ \g\in \G_{n-1},\g\ne 1\}$ are both disjoint from $\E(x^-)$.
    As $x^\pm\in \C(\Lambda_{\G_{n-1}})$, by \Cref{lem:contraction}, we have 
    \[\overline{\G_{n-1}\cdot x^\pm} \Hquad=\Hquad (\G_{n-1}\cdot x^\pm) \sqcup \Lambda_{\G}.\]
    Moreover, since $\Lambda_{\G}\subset\C_{x^-}$, $\Lambda_{\G}$ is also disjoint from $\E(x^-)$.
    Thus, \ref{cor:precpt:i} follows.

    The proof of \ref{cor:precpt:ii} is similar to the one of \ref{cor:precpt:i}. So, we omit the details.
    
    Finally, \ref{cor:precpt:iii} is a consequence of \Cref{lem:k}.
\end{proof}

\subsection{Ping-pong}\label{sec:ping-pong}

Let $x^\pm \in\F$ be as before; see the paragraph after the proof of \Cref{lem:xpm}.
Let $B^\pm_r$ 
be closed metric balls of radii $r>0$ and with centers at $x^\pm$, respectively.
Assume that $r$ is small enough such that $B^+_r\cap B^-_r=\emptyset$ and $B^\pm_r \subset C(\Lambda_{\G_{n-1}})$. 
Let $L$  be a compact subset of $\F$ containing $\Lambda_{\G_{n-1}}$ in its interior.
We can (and will) further assume that $L$ disjoint from $B^\pm_r$ and $\E_{x^\pm}$.

Since $\G_{n-1}$ is Anosov and $(B^+_r\cup B^-_r) \subset C(\Lambda_{\G_{n-1}})$, by \Cref{lem:contraction},
$
 \g (B^+_r\cup B^-_r) \subset L
$
for all but finitely many elements in $\G_{n-1}$.
Thus, applying \Cref{cor:precpt}\ref{cor:precpt:i}, \ref{cor:precpt:ii}, we see that up to making $r$ smaller, we can ensure the following:
If 
\[
B_r \coloneqq B^+_r\cup B^-_r \quad\text{and}\quad D_r \coloneqq L \cup \bigcup_{\g\in \G_{n-1}\smallsetminus\{1\}} \g B_r,
\]
then for all $x\in B_r$ and $y\in D_r$, we have that 
$x$ and $y$ are in general position.
Furthermore, up to making $r$ even smaller, \Cref{cor:precpt}\ref{cor:precpt:iii} implies that
there exists $\epsilon>0$ and a finite subset $S\subset \G$ such that
$d_\F(k_2(\g) B^\pm_r, \, \E_{x^-}) >\epsilon$ for all $\g\in\G \smallsetminus S$. Thus,
\begin{equation}\label{eqn:ping-pong0}
    Z \coloneqq \bigcup_{\g\in\G\smallsetminus S} k_2(\g) (B^+_r\cup B^-_r) 
\end{equation}
is a precompact subset of $\C_{x^-}$.

We define 
$B = B^+_{r/2} \cup B^-_{r/2}$ and $D = D_r$. Clearly, we have that 
\begin{equation}\label{eqn:ping-pong1}
    \g B \subset D^\circ \quad \text{for all nontrivial } \g\in \G_{n-1}.
\end{equation}
Let us also pick a regular element $\beta\in A^+$ such that
\begin{equation}\label{eqn:ping-pong2}
 \beta^{\pm k} D \subset B^\circ\quad \text{for all } k\in\N
\end{equation}
(this can be done since $D$ is a compact subset of $\C_{x^+}\cap \C_{x^-}$).

Thus, applying the Combination Theorem \cite[Thm. A]{DK23}, we obtain the following:

\begin{proposition}\label{prop:anosov}
    For every $m \in \N$, the homomorphism $\rho_m: \G_{n-1} \star \Z \to G$, where $\rho_m$ is the inclusion map on the free factor $\G_{n-1}$ and maps $1 \in \Z$ to $\beta^m$, is an (injective) Anosov representation.
\end{proposition}

In other words, $\<\G_{n-1}, \beta^m\>$ is an Anosov subgroup of $G$ for all $m\in\N$ and it is isomorphic to the free group of rank $n$.

\medskip

Note that by construction, the group $\<\G_{n-1}, \beta^m\>$ satisfies items \ref{cdn:1}, \ref{cdn:3}, and \ref{cdn:4} of \Cref{cdn} for all $m\in\N$.
Therefore, what remains is to verify the other two conditions.
In \S\ref{sec:crit}, we show that 
$\<\G_{n-1}, \beta^m\>$ will satisfy \ref{cdn:2} for some large enough $m$.

\subsection{Controlling the critical exponent}\label{sec:crit}

Let $\beta\in A^+$ be the element introduced in \Cref{sec:ping-pong}.
The goal of this subsection is to prove the following result, which states that
the Anosov subgroup of $G$ generated by $\G_{n-1}$ and a {\em large enough power} of $\beta$ satisfies \Cref{cdn}\ref{cdn:2}:

\begin{proposition}[Main estimate]\label{prop:crit}
There exists $m_0\in\N$ such that for all $\a\in\Pi$, 
 \[
  \delta_\a(\G_{n}) \LE \delta_\a(\G_{n-1}) + \frac{1}{2^{n}} \LE 1-\frac{1}{2^{n}},
 \]
 where $\G_n \coloneqq \<\G_{n-1},\, \beta^{m_0} \>$.
\end{proposition}

\begin{remark}
    Our readers may observe that the fact $\G_{n-1}$ being a free group doesn't play any particular role in the proofs. In fact, some of the results in this section can be easily adapted to prove the following:

\medskip\noindent
{\bf Proposition.} {\em 
    Suppose $\Delta$ is a non-elementary\footnote{One may also allow $\Delta$ to be elementary, i.e., $\Delta\cong \Z$; however, in order for this proposition to be true in this case, it might be necessary to pass to a finite index subgroup of $\Delta$ to ensure that $\B(\Delta)$, defined by \eqref{def:badsetforGamma}, is a proper subset of $\F^2$.
    Note that if $\Delta$ is non-elementary, then by \Cref{prop:badsetnull}, $\B(\Delta)$ is always a proper subset of $\F^2$.} torsion-free Anosov subgroup of $G$ such that there exists a point in $\F$ that is in general position with respect to each point in $\Lambda_\Delta$.
    Then, there exists an element $\tau\in G$ of infinite order such that:
    \begin{enumerate}
        \item  the natural homomorphism $\Delta \star \< \tau \> \to G$ is an injective Anosov representation, and
    
        \item for all $\epsilon >0$, there exists $m_0 \in \N$ such that
        \begin{equation}\label{eqn:crit_conv}
            \delta_\a(\Delta) \Hquad<\Hquad  \delta_\a(\<\Delta, \tau^m \>) \Hquad<\Hquad \delta_\a(\Delta) +\epsilon
        \end{equation}
        for all $m\ge m_0$ and all $\a\in\Pi$.\footnote{The proof of the upper inequality in \eqref{eqn:crit_conv} is similar to that of \Cref{prop:crit}.
        That the lower inequality in \eqref{eqn:crit_conv} is strict is proved in \cite[Cor. 4.2.]{CZZ-transverse}.}
        In particular, the sequence of critical exponents $\delta_\a(\<\Delta, \tau^m \>)$ converges to $\delta_\a(\Delta)$.
    \end{enumerate}
    }
\end{remark}

Before discussing the proof of \Cref{prop:crit}, let us introduce some notation for the remainder of this section: If $ j $ is a nonzero integer, then $ \varepsilon(j) \coloneqq 1 $ if $ j $ is positive, and $ \varepsilon(j) \coloneqq -1 $ if $ j $ is negative. 

The following lemma is crucial in the proof of \Cref{prop:crit}:

\begin{lemma}[Key lemma]\label{lem:wl}
There exists a constant $C>0$ such that the following holds: 
Let $w = (\beta^{j_l} \g_{l})(\beta^{j_{l-1}} \g_{l-1})\cdots (\beta^{j_{1}} \g_{1})$ be a word in the abstract free product $\Gamma_{n-1}\star \<\beta\>$,
 where $l\in\N$, $j_i\in\Z$, $j_i\ne 0$ for $i\le l-1$, $\g_i \in \G_{n-1}$, and $\g_i \ne 1$ for $i\ge 2$.
 Then, for all $\alpha \in \Pi$,
\begingroup
\makeatletter\def\f@size{10}\check@mathfonts
 \begin{align*}
 \sum_{i=1}^l \left(|j_i|\, \a(\mu(\beta^{\varepsilon(j_i)})) + \a(\mu(\g_i))\right) -Cl
 \LE
 \a(\mu(w))
 \LE \sum_{i=1}^l \left(|j_i|\, \a(\mu(\beta^{\varepsilon(j_i)})) + \a(\mu(\g_i))\right) +Cl.
\end{align*}
\endgroup
\end{lemma}

In \S\ref{prelim:proof:lem:wl}, we establish some preliminary results before presenting the proof of \Cref{lem:wl} in \S\ref{proof:lem:wl}. After that, we prove \Cref{prop:crit} in \S\ref{proof:prop:crit}.

\subsubsection{Tits representations}\label{prelim:proof:lem:wl}
By a theorem of Tits \cite{Tits} (see also \cite[Sect. 6.8 \& 6.9]{Benoist-Quint-Book}), for each $\alpha\in\Pi$, there exists a finite-dimensional irreducible representation 
$ \rho_\a : G \to {\rm GL}(V_\a)$
whose highest $\R$-weight $\rchi_\a$ is an integral multiple of the fundamental weight $\varpi_\a$ associated to $\a$,
and the corresponding highest weight space is one-dimensional.
In particular, $\{\rchi_\a :\ \a\in\Pi\}$ forms a basis of the dual space $\liea^*$.
We fix such a representation $(\rho_\a, V_\a)$ of $G$ for each $\alpha\in\Pi$.

Now, we fix $\a\in\Pi$.
Let $V_\a^+ \subset V_\a$ be the one-dimensional highest weight space of $(\rho_\a, V_\a)$ and let $V_\a^{<}$ be the unique complementary $\rho_\a(A)$-invariant subspace of $V_\a$.
We equip $V_\a$ with a 
{\em good norm}\footnote{See \cite[Lem. 6.33]{Benoist-Quint-Book}.} $\|\cdot\|$, i.e., a $\rho_\a(K)$-invariant Euclidean norm in $V_\a$ such that for all $a\in A$, $\rho_\a(a)$ is a symmetric endomorphism.
Under this norm, $V^<_\a$ is the orthogonal complement of $V_\a^+$.

Consider the map $G \to \PP(V_\a)$, $g \mapsto \rho_\a(g) V_\a^+$.
The $G$-stabilizer of $V_\alpha^+$ is the parabolic subgroup $P_\alpha$.
Thus, the map $G \to \PP(V_\a)$ factors through a $\rho_\a$-equivariant smooth embedding of the flag variety $\F_\a \coloneqq G/P_\alpha$ into the projective space
$\PP(V_\a)$ given by
$
 \iota_\a : \F_\a \to \PP(V_\a),\ [P_\alpha]\mapsto V_\a^+.
$
Moreover, the $G$-stabilizer of $V^<_\a$ is $P_\a^-$, the maximal parabolic subgroup of $G$ opposite to $P_\a$. See \S\ref{sec:notation} for notation.

\begin{lemma}\label{lem:opnorm0}
    For every $g\in G$, the operator norm of $\rho_\a (g) : V_\a\to V_\a$ is given by $\| \rho_\a (g)\| = \exp \left( \rchi_\a (\mu (g)) \right)$.
\end{lemma}

See \cite[Lem. 6.33]{Benoist-Quint-Book} for a proof of this lemma.

\begin{lemma}\label{lem:opnorm}
 For all $\epsilon>0$, there exists $C>0$ such that the following holds:
 For $g \in G$,
 write $g = k_1(g)a(g)k_2(g)$, where $a(g)\in A^+$ and $k_1(g),k_2(g) \in K$.
 If $v\in V_\a$ is a nonzero vector such that $\angle(\rho_\a(k_2(g)) v,  V_\a^<) >\epsilon$,
 then $\| \rho_\a (g) v\| \ge C \, \| \rho_\a (g)\|\, \|v\|.$
\end{lemma}

\begin{proof}
 For $w\in V_\a$, let $w^+$ (resp. $w^<$) denote the the orthogonal projection of $w$ into $V^+_\a$ (resp. $V_\a^<$).
 Then, there exists a constant $C>0$ depending only on $\epsilon>0$ such that for all nonzero $w\in V$ with $\angle(w,V_\a^<) > \epsilon$, we have that $\|w^+\| \ge C\, \|w\|$.
 Thus, for all $a\in A^+$, we get
 $
  \|\rho_\a(a) w\| = \|\rho_\a(a)(w^++w^<)\| \ge \|\rho_\a(a)(w^+)\|
  = \|\rho_\a(a)\| \, \|w^+\| \ge C\, \|w\|.
 $
 The lemma follows from this.
\end{proof}

Let $\pi_{\a} : \F \to \F_{\a}$ be the $G$-equivariant projection. The composition $\iota_{\a} \circ \pi_{\a} : \F \to \PP(V_{\a})$ maps $\C(x^{-})$ into $\PP(V_{\a}) \smallsetminus \PP(V_{\a}^<)$. This can be seen as follows: The maximal parabolic subgroup $P_{\a}^-$ opposite to $P_{\a}$ acts transitively on $\pi_{\a}(\C(x^-)) \subset \F_{\a}$. Therefore, $\iota_{\a} (\pi_{\a}(C(x^-))) = \rho_{\a}(P_{\a}^-) \cdot \iota_{\a}(\pi_{\a}(x^+)) = \rho_{\a}(P_{\a}^-) \cdot V_{\a}^+$. Since $V_{\a}^<$ is preserved by $\rho_{\a}(P_{\a}^-)$, we have
$
\iota_{\a}(\pi_{\a}(C(x^-))) \cap \PP(V_{\a}^<)
= (\rho_{\a}(P_{\a}^-) \cdot V_{\a}^+) \cap \PP(V_{\a}^<) 
= \rho_{\a}(P_{\a}^-) \cdot (V_{\a}^+ \cap \PP(V_{\a}^<))
= \emptyset.
$
So, $\iota_{\a} ( \pi_{\a} (\C(x^-)) ) \subset \PP(V_{\a}) \smallsetminus \PP(V_{\a}^<)$.

Recall the compact sets $B, D, Z \subset \F$ defined in \Cref{sec:ping-pong}. 
Let 
$B_{\a} \coloneqq \iota_{\a} ( \pi_{\a} (B))$,
$D_{\a} \coloneqq \iota_{\a} ( \pi_{\a} (D))$,
and
$Z_{\a} \coloneqq \iota_{\a} ( \pi_{\a} (Z))$.
By \eqref{eqn:ping-pong1} and \eqref{eqn:ping-pong2}, we obtain that for all nontrivial $\g \in \G_{n-1}$ and all nonzero $j\in\Z$,
\begin{equation}\label{eqn:ping-pong3}
    \rho_{\a}(\g)B_{\a} \subset D_{\a}
    \quad\text{ and }\quad
    \rho_\a(\beta)^{j} D_\a \subset B_\a.
\end{equation}
Moreover, there exists a finite subset $S \subset \G$ such that $\rho_{\a}(k_2(\g)) B_{\a} \subset Z_{\a}$ for all $\g \in \G \smallsetminus S$; cf. \Cref{sec:ping-pong}. 
Since $Z$ is precompact in $\C_{x^-}$, by the preceding paragraph, $Z_{\a}$ is a precompact subset of $\PP(V_{\a}) \smallsetminus \PP(V_{\a}^<)$.
So, there exists $\epsilon_1 > 0$ such that
\begin{equation}\label{eqn:kgboundedaway}
    \angle(\rho_{\a}(k_2(\g)) v, \, V_{\a}^<) \geq \epsilon_1 \qquad 
    \begin{array}{c}
        \text{for all } \g \in \G \smallsetminus S \text{ and for all}\\
    \text{nonzero } v\in V_\a \text{ such that } [v] \in B_{\a}.
    \end{array}
\end{equation}
We use these facts in the proof of the following:

\begin{lemma}\label{lem:est1}
 There exists a constant $C >0$ such that for all 
 nontrivial elements $\g\in\G_{n-1}$, all $j\in\N$, and all nonzero $v\in V_\a$ with $[v]\in B_\a$,
\begin{align*}
 \| \rho_\alpha(\beta^j\g) v \| &\GE C\, \| \rho_\a (\beta)\|^j\,  \| \rho_\a (\g)\| \, \|v\|\\
 \| \rho_\alpha(\beta^{-j}\g) v \| &\GE C\, \| \rho_\a (\beta^{-1})\|^j \,  \| \rho_\a (\g)\| \, \|v\|.
\end{align*}
\end{lemma}

\begin{proof}
 Note that as $[v]\in B_\a$ and $\g\in\G_{n-1}$ is nontrivial, by \eqref{eqn:ping-pong3}, we get $\rho_\a(\g) [v]\in D_\a$.
 Since $D_\a$ is uniformly bounded away from $\PP(V_\a^<)$, there exists a constant $\epsilon_0>0$ such that
$\angle(\rho_\alpha(\g) v, V_\alpha^< ) > \epsilon_0$ for all nontrivial $\g\in\G$.
 So, by \Cref{lem:opnorm}, there exists a constant $C_0 = C_0(\epsilon_0)>0$ such that
 $
  \| \rho_\alpha(\beta^j\g) v \| 
  = \|\rho_\alpha(\beta^j)(\rho_\alpha(\g) v) \| 
  \ge C_0 \, \| \rho_\a (\beta)\|^j\, \|\rho_\alpha(\g) v\|
 $
 for all $j\in \N$.
 Moreover, since $[v]\in B_\a$, by \eqref{eqn:kgboundedaway}, we also have that $\angle(\rho_{\a}(k_2(\g)) v, \, V_{\a}^<) \ge \epsilon_1$ for all $\g\in\G\smallsetminus S$.
 So, by \Cref{lem:opnorm}, there exists a constant $C_1 = C_1(\epsilon_1)>0$ such that for all $\g\in\G_{n-1}\smallsetminus S$,
 $
  \|\rho_\alpha(\g) v\| \ge C_1\, \| \rho_\a(\g)\|\, \|v\|.
 $
  Furthermore, let 
  \[
   C_2 \coloneqq \min \left\{ 
   \frac{\|\rho_\alpha(\g) w\|}{\| \rho_\a(\g)\|\, \|w\|} :\
   \g\in S, w\in V_\a, w\ne 0
   \right\} >0.
  \]
  So, for all $\g \in \G$, we have
  $
   \|\rho_\alpha(\g) v\| \ge C_3\,  \| \rho_\a(\g)\|\, \|v\|,
  $
  where $C_3 \coloneqq \min\{C_1,C_2\} >0$.
  Thus, we get  
  $
  \|\rho_\alpha(\g) v\| \ge C_0C_3 \, \| \rho_\a (\beta)\|^j\, \| \rho_\a(\g)\|\, \|v\|,
  $
  which finishes the proof of the first inequality with $C\coloneqq C_0C_3$.

  The proof of the second inequality is similar, so we omit the details.
  \end{proof}

Recall our convention that if $ j $ is a nonzero integer, then $ \varepsilon(j) = 1 $ if $ j $ is positive, and $ \varepsilon(j) = -1 $ if $ j $ is negative.

\begin{lemma}\label{lem:est2}
 There exists a constant $C_\a \ge 0$ such that 
 for all $w \in \Gamma_{n-1}\star \<\beta\>$ of the form given by  \Cref{lem:wl},
 we have that
\begingroup
\makeatletter\def\f@size{10}\check@mathfonts
\begin{align}\label{eqn:est2}
 \sum_{i=1}^l \left(|j_i|\, \rchi_\a(\mu(\beta^{\varepsilon(j_i)})) + \rchi_\a(\mu(\g_i))\right) -C_\a l
 \LE \rchi_\a(\mu(w))
 \LE \sum_{i=1}^l \left(|j_i|\, \rchi_\a(\mu(\beta^{\varepsilon(j_i)})) + \rchi_\a(\mu(\g_i))\right).
\end{align}
\endgroup
\end{lemma}

\begin{proof}
 We begin with the proof of the lower inequality:
 Let  $w\in \Gamma_{n-1}\star \<\beta\>$ be an element of the form 
$
   w  = (\beta^{j_l} \g_{l})(\beta^{j_{l-1}} \g_{l-1})\cdots (\beta^{j_{1}} \g_{1}),
$
 where $l\in\N$, $j_i\in\Z$, $j_i\ne 0$ for $i\le l-1$, $\g_i \in \G_{n-1}$, and $\g_i \ne 1$ for $i\ge 2$.
 For simplicity, we further assume that $j_l\ne 0$ and $\g_1 \ne 1$. (The general case would then follow from this special case after multiplying $w$ from the right (resp. left) by some fixed nontrivial element $\g\in \G_{n-1}$ (resp. $\beta$), if needed.)
 Let $v_0\in V_\a$ be a nonzero vector such that $[v_0]\in B_\a$.
 For each $q=1,\dots,l-1$,
 let
 $
  v_q \coloneqq (\beta^{j_{q}} \g_{q})\cdots (\beta^{j_{1}} \g_{1}) v_0.
 $
 It follows from \eqref{eqn:ping-pong3} that $[v_q] \in B_\a$.
 Thus, by \Cref{lem:est1}, we get
 $
  \| \rho_\a(\beta^{j_{q+1}}\g_{q+1})v_q \|
  \ge
   C \, \| \rho_\a (\beta^{\varepsilon(j_{q+1})})\|^{|j_{q+1}|}\,  \| \rho_\a (\g_{q+1})\| \, \| v_q\|,
 $
 for some $C>0$.
 By an induction on $q$, we get
 $
  C^l\,\| v_0\|\,   \left(\prod_{i=1}^l \| \rho_\a (\beta^{\varepsilon(j_{i})})\|^{|j_{i}|}\,  \| \rho_\a (\g_{i})\|\right)
  \le
  \| \rho_\a(w)v_0 \|
 $.
 Moreover, $\| \rho_\a(w)v_0 \| \le \| \rho_\a(w) \|\, \|v_0\|$. So,
 \[
  C^l\,   \prod_{i=1}^l \| \rho_\a (\beta^{\varepsilon(j_{i})})\|^{|j_{i}|}\,  \| \rho_\a (\g_{i})\| \LE \| \rho_\a(w)\|.
 \]
 \Cref{lem:opnorm0} states that for all $g\in G$, $\| \rho_\a (g)\| = \exp \left( \rchi_\a (\mu (g)) \right)$.
 Thus, the lower inequality in \eqref{eqn:est2} follows by taking the logarithm of the above inequality and by setting $C_\a \coloneqq  \max\{0,-\log(C)\}$.

 For the upper inequality in \eqref{eqn:est2}, notice that $ \| \rho_\a(gh) v \| \le \| \rho_\a(g)\|\, \|\rho_\a(h)\| \, \| v \| $ for all $ v\in V_\a $ and all $ g,h\in G $. 
 Thus, for any nonzero vector $ v\in V_\a $ satisfying $ \| \rho_\a(w) v \| = \| \rho_\a(w) \|\, \|v\| $, we have
$
     \| \rho_\a(w) \|\, \|v\| = \| \rho_\a(w) v \|
     \le \| v\| \left(\prod_{i=1}^l \| \rho_\a (\beta^{\varepsilon(j_{i})})\|^{|j_{i}|}\,  \| \rho_\a (\g_{i})\|\right).
$
So,
\[
\begin{aligned}
     \| \rho_\a(w) \| \LE \prod_{i=1}^l \| \rho_\a (\beta^{\varepsilon(j_{i})})\|^{|j_{i}|}\,  \| \rho_\a (\g_{i})\|.
\end{aligned}
\]
Taking the logarithm of the above and using the fact that $ \| \rho_\a (g)\| = \exp \left( \rchi_\a (\mu (g)) \right) $ for $ g\in G $ again, the upper inequality in \eqref{eqn:est2} follows.
\end{proof}

\subsubsection{Proof of \Cref{lem:wl}} \label{proof:lem:wl}

Since $ \{ \rchi_\varphi :\ \varphi \in \Pi \} $ spans $ \mf{a}^* $, we can write each $ \alpha \in \Pi $ as
\begin{equation}\label{eqn:lca}
    \alpha = \sum_{\varphi \in \Pi} \lambda_{\alpha}^{\varphi} \rchi_{\varphi}
\end{equation}
for some $ \lambda_{\alpha}^{\varphi} \in \R $. Moreover, by the upper and lower bounds obtained by \Cref{lem:est2}, we obtain that for all $ \lambda \in \R $, all $ \varphi \in \Pi $, and all $ w \in \Gamma_{n-1} \star \< \beta \> $ of the form given by \Cref{lem:wl},
\begin{equation*}
    \sum_{i=1}^l \left(|j_i|\, (\lambda \rchi_\varphi)(\mu(\beta^{\varepsilon(j_i)})) + (\lambda \rchi_\varphi)(\mu(\gamma_i))\right) -|\lambda| C_\varphi l \LE \lambda \rchi_\varphi(\mu(w)).
\end{equation*}
Therefore, by \eqref{eqn:lca}, we obtain the lower inequality in \Cref{lem:wl}:
\[
 \sum_{i=1}^l \left(|j_i|\, \alpha(\mu(\beta^{\varepsilon(j_i)})) + \alpha(\mu(\gamma_i))\right) -C l \LE \alpha(\mu(w))
\]
where $ C \coloneqq \max \{ \sum_{\varphi \in \Pi} |\lambda_{\alpha}^{\varphi}| C_\varphi :\ \alpha \in \Pi \} $. 

The proof of the upper inequality in \Cref{lem:wl} is similar; so we skip the details. \qed

\subsubsection{Proof of \Cref{prop:crit}}\label{proof:prop:crit}
 
 For each $N\in\N$, let
 \begin{equation}\label{eqn:mn:proof:prop:crit}
     m_N \coloneqq  \left\lfloor\frac{N}{\min\{ \alpha(\mu(\beta)),\, \alpha(\mu(\beta^{-1})) :\ \a\in \Pi \}}\right\rfloor.
 \end{equation}
 By definition, we have that for all $\alpha\in\Pi$,
 \begin{equation}\label{eq1:proof:prop:crit}
  N \le m_N\, \alpha(\mu(\beta)) 
  \quad\text{and}\quad
  N \le m_N\, \alpha(\mu(\beta^{-1})). 
 \end{equation}
 
 Now, we fix $\alpha\in\Pi$.
 Any element $w\in \G_{n-1}\star \<\beta^{m_N}\>$ has a unique {\em reduced form} given by
 \begin{equation}\label{eqn:reduced}
     w
     =(\beta^{m_Nj_l} \g_{l})(\beta^{m_Nj_{l-1}} \g_{l-1})\cdots (\beta^{m_Nj_{1}} \g_{1})
 \end{equation}
 where $l\in \N$, $j_i\in\Z$, $j_i\ne 0$ for $i\le l-1$, $\g_i \in \G_{n-1}$, and $\g_i \ne 1$ for $i\ge 2$.
 For $N, R\in\N$, let 
 \[S_{\a,N,R} \coloneqq \{w\in \G_{n-1}\star \<\beta^{m_N}\> :\ \a(\mu(w)) \le R \}.\]
 We will estimate the number of elements in $S_{\a,N,R}$.
 If $w \in S_{\a,N,R}$ is any word whose reduced form is given by \eqref{eqn:reduced}, then by \Cref{lem:wl}, we obtain
\begin{equation}\label{eqn:R}
   \sum_{i=1}^l \left(m_N |j_i|\, \a (\mu(\beta^{\varepsilon(j_i)})) + \a(\mu(\g_i))\right) -Cl \LE R,
\end{equation}
where $C>0$ is the constant given by \Cref{lem:wl}.
By \eqref{eq1:proof:prop:crit}, we get
$
 \sum_{i=1}^l N |j_i| - Cl \le R.
$
Since $j_i\ne 0$ for $i\ge 2$, we thus obtain $N(l-1) - Cl \le R$. So, if $N\ge C+1$, then
\begin{equation}\label{eqn:lbound}
  l \le \frac{N+R}{N-C}.
\end{equation}
 Therefore, by \eqref{eqn:R}, we get
\begin{equation}\label{eqn:lbound2}
 \sum_{i=1}^l m_N |j_i|\, \a (\mu(\beta^{\varepsilon(j_i)}))
 \le C_{N,R}
 \quad\text{and}\quad
 \sum_{i=1}^l \a(\mu(\g_i))
 \le C_{N,R},
\end{equation}
where
\begin{equation}\label{eq1:lem:finalestimate}
    C_{N,R} \coloneqq \left\lceil R+\frac{C(N+R)}{N-C}\right\rceil.
\end{equation}
Let
\begin{align*}
 S_{\a,N,R,l}' &\coloneqq \biggl\{ (\g_1,\dots,\g_l) \in (\G_{n-1})^l :\ \sum_{i=1}^l \a(\mu(\g_i)) \le C_{N,R} \biggr\},\\
 S_{\a,N,R,l}'' &\coloneqq \biggl\{ (j_1,\dots,j_l) \in \Z^l :\ \sum_{i=1}^l m_N|j_i| \,\a(\mu(\beta^{\varepsilon(j_i)})) \le C_{N,R} \biggr\}.
\end{align*}
It follows from \eqref{eqn:lbound} and \eqref{eqn:lbound2} that
\begin{equation}\label{eqn:SR}
  |S_{\a,N,R}| \LE
 \sum_{l=1}^{\left\lfloor\frac{N+ R}{N-C}\right\rfloor}  |S_{\a,N,R,l}'|\, | S_{\a,N,R,l}''|.
\end{equation}

\begin{lemma}\label{lem:finalestimate}
 For all $\epsilon\in(0,1)$, there exists an integer $N_\alpha \ge C+1$ such that  for all $N\ge N_\a$, all large $R$ (depending on $N$), and all $l\le (N + R)/(N-C)$,
\begin{align}
   |S_{\a,N,R,l}'| &\LE  \exp \bigl( R(\delta_\a(\G_{n-1}) +2\epsilon/3)\bigl), \label{eqn1:lem:finalestimate}\\
   |S_{\a,N,R,l}''| &\LE \exp \bigl( R\epsilon /3\bigl).\label{eqn2:lem:finalestimate}
\end{align}
\end{lemma}

We will prove the above two inequalities separately. During the proofs, we may obtain two different values for $N_\a$, but taking the maximum of these two would satisfy the hypothesis of the lemma.

\begin{proof}[Proof of the first inequality \eqref{eqn1:lem:finalestimate} in \Cref{lem:finalestimate}]
Observe that the number of $l$-tuples of non-negative integers $(q_1, q_2, \dots, q_l)$  such that $\sum_{i=1}^l q_i = C_{N,R}$ is equal to $\binom{C_{N,R} + l -1}{l-1}$.
Moreover, given such a tuple $(q_1, q_2, \dots, q_l)$,
the set of all elements $(\gamma_1, \gamma_2, \dots, \gamma_l) \in (\G_{n-1})^l$ such that $\alpha(\mu(\gamma_i)) \leq q_i$ is at most $\prod_{i = 1}^l De^{(\delta_\a(\Gamma_{n-1}) + \epsilon/2) q_i}$, where $D \ge 1$ is some uniform constant depending only on $\epsilon$.
Therefore, we have
\begin{equation}\label{eqn:lem:finalestimate}
    |S_{\a,N,R,l}'|  \LE  {\binom{C_{N,R} + l -1 }{l -1}}D^l \exp\left({\left(\delta_\a(\G_{n-1}) +\frac{\epsilon}{2} \right) C_{N,R}}\right) .
\end{equation}
Since $\binom{C_{N,R} + l -1}{l-1} =\binom{C_{N,R} + l -1}{C_{N,R}} $ is strictly increasing with $l$ and $l\le \left\lfloor\frac{N+R}{N-C}\right\rfloor$, we have
\begin{align}\label{eqn2:ineq:lem:finalestimate}
    \binom{C_{N,R} + l -1 }{l-1}D^l
    &\LE \binom{C_{N,R} + \left\lfloor\frac{N+R}{N-C}\right\rfloor}{ \left\lfloor\frac{N+R}{N-C}\right\rfloor }D^{\frac{N+R}{N-C}} 
    \LE  \left(De\frac{C_{N,R} + \frac{N+R}{N-C}}{\frac{N+R}{N-C}-1 }\right)^{\frac{N+R}{N-C}}
\end{align}
where in the last inequality we have used the fact that $\binom{P}{Q} \le \left(\frac{eP}{Q}\right)^Q$.
A direct computation using \eqref{eq1:lem:finalestimate} shows
\begin{equation*}
    De\frac{C_{N,R} + \frac{N+R}{N-C}}{\frac{N+R}{N-C}-1 } \to De(N+1)
    \quad\text{as } R\to\infty.
\end{equation*}
Moreover,
\begin{equation*}
    \frac{1}{R} \frac{N+R}{N-C} \to \frac{1}{N-C}
    \quad\text{as } R\to\infty.
\end{equation*}
Therefore, for all $N$, there exists $R_N$ such that for all $R\ge R_N$, we have
\begin{equation*}
    De\frac{C_{N,R} + \frac{N+R}{N-C}}{\frac{N+R}{N-C}-1 } \LE De(N+2)
    \quad\text{and}\quad
    \frac{N+R}{N-C} \LE  \frac{R}{N-C} \left(1+\frac{\epsilon}{2024} \right).
\end{equation*}
Putting these bounds together,
by \eqref{eqn2:ineq:lem:finalestimate}, we have
\begin{align*}
    \binom{C_{N,R} + l -1}{l-1}D^l
    &\LE \left((De(N+2))^{\frac{1}{N-C}\left(1+\frac{\epsilon}{2024} \right)}\right)^{R}.
\end{align*}
Taking $N$ large enough we can make the quantity $\left((De(N+2))^{\frac{1}{N-C}\left(1+\frac{\epsilon}{2024} \right)}\right)$ as close to $1$ as we want; so we will take large enough $N_\a \in\N$, so that for all $N\ge N_\a$, we have
\begin{align}\label{eqn:eby7}
    \binom{C_{N,R} + l -1}{l-1}D^l
    &\LE e^{\frac{R\epsilon}{7}}.
\end{align}
for all $R\ge R_N$.
If needed, we may choose $N_\a$ even larger so that for all $N\ge N_\a$, we have
\begin{equation*}
    C_{N,R} \LE R+ \frac{R\epsilon}{2024}.
\end{equation*}
So, by \eqref{eqn:lem:finalestimate}, we get that for all $N\ge N_\a$ and $R \ge R_N$, 
\begin{align*}
    |S_{\a,N,R,l}'| &\LE \exp\left(\frac{R\epsilon}{7}+ \left(\delta_\a(\G_{n-1}) +\frac{\epsilon}{2} \right)C_{N,R}\right)\\
    &\LE \exp\left(\frac{R \epsilon}{7}+  \left(\delta_\a(\G_{n-1}) +\frac{\epsilon}{2} \right)\left(R+ \frac{R\epsilon}{2024}\right)\right)
    \LE \exp \left(R \delta_\a(\G_{n-1}) + \frac{2\epsilon R}{3}\right)
\end{align*}
where the last inequality follows by the fact that $\epsilon$ and $\delta_\a(\G_{n-1})$ are both less than $1$.
This proves the first inequality in \Cref{lem:finalestimate}.
\end{proof}

\begin{proof}[Proof of the second inequality \eqref{eqn2:lem:finalestimate} in \Cref{lem:finalestimate}]
Note that $S_{\a,N,R,l}''$ is contained in the ball of radius 
$ {C_{N,R}}/\left( m_N \min \{ \alpha(\mu(\beta)),\, \alpha(\mu(\beta^{-1})) \}\right)$
in $\R^l$ (with $L^1$-norm) centered at the origin.
Moreover, using the expression of $m_N$ \eqref{eqn:mn:proof:prop:crit}, we see that the quantity above does not exceed $C_{N,R}/N$.
Since the number of integer solutions $(j_1,\dots,j_l)$ to the equation $\sum_{i=1}^l |j_i| = K$ is at most $\binom{K + l -1}{l-1} 2^l $, we obtain
 \begin{equation*}
     |S_{\a,N,R,l}''| \LE  \sum_{i=0}^{\lfloor C_{N,R}/N \rfloor}\binom{i + l -1}{l-1} 2^l \LE  \left( \frac{C_{N,R}}{N} +1 \right)\binom{{\lfloor C_{N,R}/N \rfloor} + l -1}{l-1} 2^l.
 \end{equation*}
Imitating the proof of the inequality \eqref{eqn:eby7},
we get that there exists $N_\a\in\N$ such that for all $N\ge N_\a$ and all large $R$ (depending on $N$),
$
     |S_{\a,N,R,l}''| \LE e^{\frac{R\epsilon}{3}}.
$
\end{proof}

Let $N_0 \coloneqq \max\{ N_\a :\ \a\in\Pi\}$, where $N_\a\in\N$ is the constant given by \Cref{lem:finalestimate}.
Using \eqref{eqn:SR} and applying the \Cref{lem:finalestimate}, we thus obtain that for all large enough $R\in\N$ and all $\a\in\Pi$,
\begin{align}\label{eqn:crit:final}
\begin{split}
     |S_{\alpha,N_0,R}| \LE
 \sum_{l=1}^{\left\lfloor\frac{N_0+R}{N_0-C}\right\rfloor} 
  \exp \bigl( (R(\delta_\a(\G_{n-1}) +\epsilon)\bigl)
 \LE
 \left(\frac{N_0+R}{N_0-C}\right) \exp \bigl( R(\delta_\a(\G_{n-1}) +\epsilon)\bigl).
\end{split}
\end{align}

Set $\epsilon = 1/2^n$, $m_0 \coloneqq m_{N_0}$, and let
$
    \G_{n} \coloneqq \<\G_{n-1},\, \beta^{m_0}\> .
$
Applying \eqref{eqn:crit:final}, we obtain that for all $\a\in\Pi$,
\[
 \delta_\a\left(\G_{n}\right) \Hquad=\Hquad \limsup_{R\to\infty}\frac{\log |S_{\alpha,N_0,R}|}{R} \LE \delta_\a(\G_{n-1}) +\frac{1}{2^n}.
\]
This completes the proof of \Cref{prop:crit}. \qed

\subsection{Proof of the inductive step}\label{sec:pf_induction}
Let $x_+\in\F$ be as in \Cref{sec:xpm}, $\beta\in A^+$ be as in \Cref{sec:ping-pong}, and $m_0$ and $\G_n = \< \G_{n-1}, \beta^{m_0}\>$ be as in \Cref{prop:crit}. Define
$\beta_n \coloneqq \beta^{q_0m_0}$ and $\quad x_n \coloneqq x^+$.
We verify that $\G_n$ satisfies \Cref{cdn}:
By \Cref{prop:anosov}, $\G_n$ is Anosov and freely generated by $\G_{n-1}$ and $\beta_n$. 
By \Cref{prop:crit}, $\delta_a(\G_n) \le 1-\frac{1}{2^n}$.
Finally, by \eqref{eqn:xpm}, $d_\F(x_n,z_n) \le \frac{1}{n}$.

The proof of the inductive step is now complete. \qed

\section{Proof of Theorem \ref*{thm:main}}

We are now in a position to finish the proof of \Cref{thm:main}.

\begin{proof}[Proof of \Cref{thm:main}]
Let $ G $ be a non-compact connected real semi-simple Lie group with finite center. In \Cref{mainconstruction}, we constructed an increasing sequence $ (\G_n) $ of discrete subgroups of $ G $ satisfying \Cref{cdn}. Define a subgroup of $G$ by
\begin{equation}\label{eqn:groupGamma}
    \G = \bigcup_{n\in\N} \G_n.
\end{equation}
Since, by \Cref{cdn}\ref{cdn:3}, $ \G_n $ is freely generated by $ \{\beta_1,\dots,\beta_n\} $ for each $ n $,
 it follows that $ \G $ is freely generated by $ \{\beta_n:\ n\in\N \} $.

To show that $\G$ is discrete, we apply:

\medskip \noindent
{\bf Margulis lemma} (see \cite[Thm. 9.5]{BGS}){\bf .}
{\em Suppose that $(X,d)$ be a simply-connected, complete Riemannian manifold with non-positive bounded sectional curvature. There exists a constant $\epsilon>0$ such that the following holds: If $\Gamma$ is a discrete group of isometries of $X$ and $p\in X$,
    then the group generated by $\{\gamma\in \Gamma :\ d(p,\gamma p) <\epsilon\}$ is virtually nilpotent. }
\medskip

We show that the group $\G$ given by \eqref{eqn:groupGamma} is a discrete subgroup of $G$: To the contrary, if $\G$  is indiscrete, 
then it contains a sequence $(\g_m)$ of pairwise distinct elements converging to the identity element $1\in G$.
Let $X$ be the symmetric space of $G$ and let us fix a basepoint $p\in X$.
So, 
\begin{equation}\label{eqn:convergingtozero}
    d(p,\g_mp) \to 0 \quad\text{as } m\to\infty.
\end{equation}
By the Margulis lemma, it follows that for all sufficiently large integers $m_1$ and $m_2$, $\g_{m_1},\g_{m_2}$ lie in a virtually nilpotent subgroup of $G$ (since $\g_{m_1},\g_{m_2}$ lie in some $\G_n$, which is discrete).
Since the only nontrivial virtually nilpotent subgroup of a free group is $\Z$, $\g_{m_1},\g_{m_2}$ must commute for all large $m_1,m_2\in\N$.
Together with the fact that $\G$ is the union of an increasing sequence of free groups $(\G_n)$, it follows that the entire sequence $(\g_m)$ must lie in some $\G_{n_0}$.
However, in that case, \eqref{eqn:convergingtozero} implies that $\G_{n_0}$ is indiscrete, which is a contradiction.

Finally, according to \Cref{cdn}\ref{cdn:4}, the attractive fixed points $ \{x_n:\ n\in\N\} $ in the Furstenberg boundary $ \F $ of $ G $ for elements of $ \G_n $ form a dense set. Therefore, $ \Lambda_{\G} = \F $.
This concludes the proof of \Cref{thm:main}.
\end{proof}

A minor modification of the above proof shows that the analog of Theorem \ref{thm:main} holds in more constrained situations as follows:

\begin{theorem}\label{mainthm_lattice}
    Let $\Delta$ be a Zariski-dense subgroup of a connected semisimple real algebraic group $G$ acting minimally on the Furstenberg boundary $\F = G/P$ (e.g., $\Delta$ is a lattice).
    Then, there exists an infinitely generated discrete subgroup $\G$ of $G$, contained in $\Delta$, with a full limit set in $\F$.
\end{theorem}

\begin{proof}[Proof sketch]
    To construct $\G$, one may follow the same inductive construction as described in Section \ref{mainconstruction} with the following additional observations: Since $\Delta$ is Zariski-dense, by \cite[Thm. 6.36]{Benoist-Quint-Book}, the set of loxodromic elements in $\Delta$ is Zariski-dense in $G$. Then, the action of $\Delta$ on $\F$ is proximal; that is, for any $x,y\in\F$ and $\epsilon>0$, there exists $g\in \Delta$ such that $d_\F(gx,gy) < \epsilon$. Together with the fact that $\Delta$ acts minimally on $\F$, it follows that for any nonempty open set $U\subset \F$, one can find a loxodromic element in $\Delta$ with attractive/repulsive fixed points lying in $U$. At the $n$-th step of the inductive construction, using these facts, one can find a loxodromic element $\beta\in\Delta$ whose attractive/repulsive fixed points $x^\pm$ satisfy condition \ref{lem:xpm:cdn2} in \Cref{lem:xpm} as well as \eqref{eqn:xpm}. The rest of the argument works as before.
\end{proof}

\section{\texorpdfstring{Some criteria for full limit set in $\SL(3,\R)$}{}}\label{sec:criterion}

The goal of this section is to provide some practical criteria for determining whether a Zariski-dense subgroup of $\SL(3,\Z)$ has a full limit set in the Furstenberg boundary of $\SL(3,\R)$. Throughout this section, we denote $G = \SL(3,\R)$. We prove:

\begin{theorem}\label{sl3full} Suppose that a subgroup $\Gamma \subset  \SL(3,\Z)$, which is Zariski-dense in $G$, satisfies one of the following:
\begin{enumerate}
\item $\Gamma$ contains a subgroup isomorphic to $\Z^2$.\label{sl3full:case1}
\item $\Gamma$ contains a standard copy of $\SL(2,\Z)$.\label{sl3full:case2}
\item $\Gamma$ contains an infinite order element with at least one complex (non-real) eigenvalue.\label{sl3full:case3}
\end{enumerate}
Then $\Gamma$ acts minimally on $G/P$. (Equivalently, $P$ acts minimally on $\Gamma\backslash G$.) 
\end{theorem}

\begin{remark}
The places where it is used that $\Gamma < \SL(3,\Z)$ and not  another lattice are

\begin{enumerate}
\item In case \ref{sl3full:case1}, as subgroups of $\SL(3, \Z)$ isomorphic to $\Z^2$ whose non-trivial elements are  semi-simple elements do not contain singular elements, which is not true for every lattice of $G$ (there exists both uniform, and non-uniform lattices containing such elements).

\item In case \ref{sl3full:case3}, as infinite order elements with complex eigenvalues don't have powers with real eigenvalues.
\end{enumerate}
\end{remark}

Before starting with the proof, we will illustrate how \Cref{sl3full} can be used computationally to show a discrete subgroup of $G$ has full limit set; we discuss an example of a group that is currently unknown to be thin or a lattice.

Consider the following two matrices in $\SL(3,\Z)$:
 \[
 a \coloneqq \begin{bmatrix} 1&1&2\\  0&1&1\\ 0&-3&-2  
 \end{bmatrix},
 \quad b \coloneqq \begin{bmatrix} -2&0&-1 \\ -5&1&-1 \\ 3&0&1  
 \end{bmatrix}.
 \]
 Let $\Gamma$ be the subgroup of $\SL(3,\Z)$ generated by $a$ and $b$.
It is known that $\Gamma$ is Zariski-dense in $\SL(3,\R)$. Moreover, it is the image of a representation of the $(3,3,4)$-triangle group $\<a, b \mid  a^3 = b^3 = (ab)^4 = 1\>$ into $G$. As far as we know, it is unknown whether $\Gamma$ is of infinite index in $\SL(3,\Z)$; see Kontorovich-Long-Lubotzky-Reid \cite[Ex. 11]{kontorovich2019thin}.

 \begin{proposition}\label{prop:KLLR}
  The group $\Gamma = \< a,b\>$ as above acts minimally on $G/P$. In particular, it is not an Anosov subgroup of $G$. 
 \end{proposition}

 \begin{proof}
An extensive computation shows that the element of word length $18$ given by
\[w := ba^{-1}ba^{-1}ba^{-1}ba^{-1}b^{-1}aba^{-1}ba^{-1}b^{-1}aba^{-1}\]
has infinite order and has a pair of complex eigenvalues and, therefore, by Theorem \ref{sl3full}, $\Gamma$ has a full limit set in $G/P$.
\end{proof}

We begin with the proof of Theorem \ref{sl3full}, which is a case by case argument, and we will break it into different propositions. We will make use of the identification of $G/P$ with $\Flag(\R^3)$, the space of flags  in $\R^3$.
For convenience, a complete flag $\{0\} \subset V_1 \subset V_2\subset \R^3$ in $\R^3$ can be viewed as a pointed line $(p,l)$ in $\PP(\R^3)$, where $p = \PP(V_1)$ lies in $l = \PP(V_2)$. Furthermore, let  
\[
\pi : \Flag(\R^3) \to \PP(\R^3) \quad\text{and} \quad \pi^{*} : \Flag(\R^3) \to \PP((\R^3)^*)
\]be the natural $G$-equivariant projections.

We will make use extensively of the following proposition to deal with most of the cases of Theorem \ref{sl3full}. 

\begin{proposition}\label{samicriterion}  Let $\Gamma$ be a discrete, Zariski-dense subgroup of $G = \SL(3,\R)$. 
Suppose the limit set $\Lambda_{\Gamma}$ contains at least one fiber of both projections $\pi$ and $\pi^*$, then $\Lambda_{\Gamma} = G/P$.
\end{proposition}

\begin{proof}
   Since $\G$ is Zariski-dense, the limit set $\Lambda_\G$ of $\G$ is the unique non-empty, minimal $\G$-invariant closed subset of $\Flag(\R^3)$ \cite{Benoist}, and the same is true for $\pi(\Lambda_{\Gamma})$ and $\pi^{*}(\Lambda_{\Gamma})$ in $\PP(\R^3)$ and $\PP((\R^3)^{*})$ respectively. We first show that $\Gamma$ acts minimally on $\PP(\R^3)$ (i.e.  $\pi(\Lambda_\G) = \PP(\R^3)$). Observe that $\pi(\Lambda_\G)$ must contain a projective line $l_0$ (by the hypothesis on $\pi^*$). Let $p_0$ be a point in $\PP(\R^3)$ such that $\pi^{-1}(p_0)$ is contained in $\Lambda_{\Gamma}$. Observe that $$\bigcup_{(p_0, l) \in \Flag(\R^3)}\{q \in \PP(\R^3 ) :\  q \in l\} = \PP(\R^3).$$ Therefore, we only need to show that if $(p_0, l) \subset \pi^{-1}(p_0)$, then $l \subset \pi(\Lambda_{\Gamma})$. From the uniqueness of the limit set $\pi^{*}(\Lambda_{\Gamma})$, it follows that for any projective line $l$ in $\pi^{*}(\Lambda_{\Gamma})$, there exists a sequence of elements $\gamma_n$ such that $\gamma_n(l_0) \to l$, and as the line $\gamma_n(l_0) \subset \pi(\Lambda_\Gamma)$, then $l \subset \pi(\Lambda_\Gamma)$. This finishes the proof of the minimality of $\Gamma$ acting on $\PP(\R^3)$.

   Let $\Lambda'_{\Gamma}$ be the set of flags $(p,l)$ such that $\pi^{-1}(\pi(p,l))$ is contained in $\Lambda_{\Gamma}$.   this is a closed non-empy $\Gamma$-invariant subset of $\Lambda_{\Gamma}$, by minimality $\Lambda'_{\Gamma} = \Lambda_{\Gamma} $, and as $\pi(\Lambda_{\Gamma}) = \PP(\R^3)$, we must have $\Lambda_{\Gamma} = \Flag(\R^3).$
\end{proof}

\subsection{Case \ref{sl3full:case3}: $\Gamma$ contains an element with complex (non-real) eigenvalues} Theorem \ref{sl3full}, case \ref{sl3full:case3}, follows immediately from the following:

\begin{proposition}\label{criterion:SL3}
    Let $\Gamma$ be a discrete, Zariski-dense subgroup of $G = \SL(3,\R)$. Suppose there exists an element $\g_0\in\Gamma$
    such that $\g_0$ has a complex eigenvalue $\lambda e^{2\pi i\theta}$ for some $\lambda>1$ and $\theta$ irrational. Then $\G$ acts minimally on $\Flag(\R^3)$.
\end{proposition}

\begin{proof}
    Let $\G$ and $\g_0$ be as described in the statement above. Then, $\g_0$ has the Jordan canonical form
    \[
    \begin{bmatrix}
        \lambda \cos\theta & -\lambda\sin\theta&\\
        \lambda\sin\theta & \lambda\cos\theta&\\
        && \lambda^{-2}\\
    \end{bmatrix}.
    \]
    Hence, $\g_0$ preserves some complementary subspaces $V$ and $W$ of $\R^3$ of dimensions $2$ and $1$, respectively. With respect to an appropriate inner product on $\R^3$, $\g_0$ acts on $V$ by rotating by an angle $\theta$ and dilating by a factor $\lambda >1$.
    Since $\theta$ is irrational, the induced actions of the semigroups $\{\g_0^n :\ n\in\N\}$ and $\{\g_0^{-n} :\ n\in\N\}$ on $l_0 \coloneqq \PP(V)$ are minimal (i.e., the orbit of any point is dense). 
    Furthermore, $\g_0$ contracts $W$ by a factor $\lambda^2 > 1$. Let $p_0 \coloneqq \PP(W)$, the unique fixed point of $\g_0$ in $\PP(\R^3)$.

    \begin{lemma}\label{lem:one}
        We have $\pi^{-1}(p_0) \subset \Lambda_\Gamma$ and $(\pi^*)^{-1}(l_0) \subset \Lambda_\Gamma$.
    \end{lemma}

    \begin{proof}
        Pick any point $(p,l)\in\Lambda_\Gamma$. 
        Since $\G$ is Zariski-dense, we may assume that $p\not\in l_0$. Then, $\gamma_0^{-n} p \to p_0$ as $n\to\infty$ and, therefore, we conclude that the sequence $(\g^{-n}(p,l))_{n\in\N}$ has an accumulation point of the form $(p_0,l)$ for some line $l\subset \PP(\R^3)$ passing through $p_0$. This point must be in $\Lambda_\Gamma$ whose $\langle \gamma_0\rangle$-orbit closure is the fiber $\pi^{-1}(p_0)$. Thus, the first conclusion follows.

        For the second conclusion, pick any point $(p,l)\in\Lambda_\Gamma$ such that $p_0\not\in l$ and apply a similar argument as above. We leave the details to the reader.
    \end{proof}

The proposition now follows from \Cref{samicriterion}.
\end{proof}

\subsection{Case \ref{sl3full:case1}: $\Gamma$ contains a subgroup isomorphic to $\Z^2$} We will continue with the proof of Theorem \ref{sl3full}. This time we consider the case where $\Gamma$ contains a copy of $\Z^2$. We will use again extensively Proposition \ref{samicriterion}.

\begin{proposition}\label{prop:sl3full:case1}
Let $G= \SL(3,\R)$ and $\Gamma \subset G$ be a Zariski-dense discrete subgroup that contains an abelian subgroup isomorphic to $\Z^2$ that does not contain any singular semi-simple elements. 
Then, $\Gamma$ acts minimally on $G/P$.
\end{proposition}

\begin{proof}

Let $\Delta \subset \Gamma$ be the subgroup of $\Gamma$ isomorphic to $\Z^2$. Using the Jordan canonical form, the discreteness of $\Gamma$, and the fact that $\Delta$ does not contain singular elements, we can assume (up to replacing $\Delta$ with a finite index subgroup), the following:

\begin{enumerate}

\item All the elements of $\Delta$ are semi-simple and $\Delta$ is conjugate to a group of diagonalizable matrices.

\item $\Delta$ consist of unipotent elements.

\item $\Delta$ is generated by two elements:

 $$\gamma_1 = \begin{bmatrix} e^t &1 &0 \\ 0 &e^{t}&0 \\0 &0&e^{-2t} 
\end{bmatrix}, \quad   \gamma_2 = \begin{bmatrix} e^s &x &0 \\ 0 &e^{s}&0 \\0 &0&e^{-2s} 
\end{bmatrix},$$ where $(t,s) \neq (0,0)$, and $x \not\in \mathbb{Q}$.

\end{enumerate}

Let us consider case (i) where all the elements are semisimple, we can assume after conjugation that $\Gamma$ contains a subgroup isomorphic to $\Z^2$ consisting of diagonal matrices. Let $A$ be the group of diagonal matrices in $G$ and let $\liea$ be its lie algebra. 
Since $\liea$ comprises $3\times 3$ traceless diagonal matrices, we will identify
$$\liea = \big\{(x_1,x_2, x_3 ) \in \R^3  :\  x_1 + x_2 + x_3 = 0\big\}$$
by sending $X\in\liea$ to $(X_{11},X_{22},X_{33})$.
We define $\Delta = A \cap \Gamma$ and $H = \log \Delta $. For $i,j \in \{1,2,3\}$ where $i \neq j$, the roots $\alpha_{i,j}: \liea \to \R$ are given by $\alpha_{i,j}(w) = x_i - x_j$, for $w = (x_1,x_2, x_3) \in \liea$.

Our assumption that all elements of $\Delta$ are nonsingular implies that the linear maps $\alpha_{i,j}: H \to \mathbb{R}$ have trivial kernel and therefore dense image. 
This will be the key to show that the limit set is equal to $\Flag(\R^3)$.

We have the following:

\begin{lemma} The limit set $\pi(\Lambda_{\Gamma})$ of $\Gamma$ in $\PP(\R^3)$ contains the lines $x= 0$, $y =0$, and $z = 0$.
\end{lemma}



\begin{proof}

We consider $\PP(\R^3)$ in the projective coordinates $[x,y,z]$. Let us take $[x_0,y_0,z_0] \in \pi(\Lambda_{\Gamma})$. By Zariski density, we can (and will) assume all the coordinates of $[x_0,y_0,z_0]$ are different from zero. Let us further assume that $ x_0,y_0, z_0 > 0.$ The other cases follow similarly. 

Consider the affine chart $U_z \coloneqq \PP(\R^3) \cap \{z = 1\}$. 
The action of $w \in \liea$ on $U_z$ in the affine coordinate $(x,y)$ is given by
\begin{equation}\label{eqn:p3min}
    e^{w}\cdot (x,y) = (e^{\alpha_{1,3}(w)}x,\, e^{\alpha_{2,3}(w)}y).
\end{equation}
By the density of the image in $\R$ of $\alpha_{i,j}$ restricted to $H$, it follows that for every $t \in \R$, there exists a sequence $w_n \in H$ such that $ \alpha_{1,3}(w_n) \to t$ and $\alpha_{2,3}(w_n) \to - \infty$.
So, we have that $e^{w_n}\cdot [x_0, y_0, z_0] \to [ e^{t} {(x_0/z_0)},0,1]$, which shows that the segment of line $l_y = \{[x,y,z] :\ x>0,\,y = 0,\, z > 0\}$ is contained in $\pi(\Lambda_{\Gamma})$.
Arguing similarly, we conclude that the other two line segments $l_x
= \{[x,y,z] :\ x = 0,\, y>0,\, z > 0\}$ and $l_z = \{[x,y,z] : x > 0,\, y>0,\, z = 0\}$ 
are also contained in $\pi(\Lambda_{\Gamma})$.
Note that $l_x$, $l_y$, and $l_z$ forms a $\Delta$-invariant projective triangle.
Using Zariski density of $\G$ and by applying a similar argument, it is not difficult to see that $\pi(\Lambda_{\Gamma})$ must contain the full three lines $x=0$, $y=0$, and $z=0$.
\end{proof}

Arguing similarly, it follows that the limit set $\pi^*(\Lambda_{\Gamma})$ of $\Gamma$ in $\PP((\R^3)^{*})$ contains three lines, and by Proposition \ref{samicriterion} we have $\Lambda_{\Gamma} = \Flag(\R^3)$ in case (i).\\

Now we can consider case (ii) for $\Delta \subset \Gamma$ (all the elements of $\Delta$ are unipotent). Up to conjugation, we can assume that  $\Delta$ is generated by $\gamma_1, \gamma_2$, where $$ \gamma_1 = \begin{bmatrix} 1 &0 &1 \\ 0 &1&0 \\0 &0&1 
\end{bmatrix}$$ and $\gamma_2$ is equal to one of the following three matrices:

$$\begin{bmatrix} 1 &1 &0 \\ 0 &1&0 \\0 &0&1 
\end{bmatrix},\quad
\begin{bmatrix} 1 &0 &0 \\ 0 &1&1 \\0 &0&1 \end{bmatrix}, \quad
\begin{bmatrix} 1 &1 &0 \\ 0 &1&x \\0 &0&1 
\end{bmatrix} \text{, where } x\neq 0.$$

In all these three cases, we can again use Proposition \ref{samicriterion}. To illustrate the proof (all of them similar), we consider the most difficult case, where $$\gamma_2 = \begin{bmatrix} 1 &1 &0 \\ 0 &1&x \\0 &0&1 
\end{bmatrix} \text{ and } x \neq 0.$$ 
In that case, we can show that  $\pi(\Lambda_{\Gamma})$ must contain the line $z =0$, as follows:

By Zariski density, there exists $x_0 \neq 0, y_0 \neq 0$ so that the projective point $ p = [x_0, y_0, 1] \in \pi(\Lambda_{\Gamma})$. We have 
$$\gamma_1^{n}\gamma_2^m (p) = [x_0 + my_0 + x \frac{m(m-1)}{2} + n, y_0 + mx, 1 ],$$
and by choosing appropriately sequences $m,n \to \infty$, we can guarantee that $\gamma_1^{n}\gamma_2^m (p)$ converge to a any given point in the line $z = 0$. Similarly, the limit set $\pi^{*}(\Lambda_{\Gamma})$ contains all lines passing through $(1,0,0)$, and therefore we can apply Proposition \ref{samicriterion} to show $\Lambda_{\Gamma} = \Flag(\R^3)$.

We should point out that if, in addition, $\Gamma \subset \SL_3(\Z)$, the results of Venkataramana \cite{Venkataramana} imply that $\Gamma$ is a lattice. And in case (i) and (ii) for the possibilities of $\gamma_2$, the results of Oh \cite{Oh_horospherical, Oh_horospherical1} imply that $\Gamma$ is a lattice (without assuming $\Gamma \subset SL_3(\Z)$).

Finally, we consider case (iii) at the beginning of the proof of Proposition \ref{prop:sl3full:case1}, where $\Delta$ is generated by 
 $$\gamma_1 = \begin{bmatrix} e^t &1 &0 \\ 0 &e^{t}&0 \\0 &0&e^{-2t} 
\end{bmatrix}, \quad   \gamma_2 = \begin{bmatrix} e^s &x &0 \\ 0 &e^{s}&0 \\0 &0&e^{-2s} 
\end{bmatrix},$$ where $(t,s) \neq (0,0)$, and $x \not\in \mathbb{Q}$. In this case, we can again use Proposition \ref{samicriterion} and show that $\pi(\Lambda_{\Gamma})$ contains the line $z = 0$ and $\pi^{*}(\Lambda_{\Gamma})$ contains all line passing through $(0,0,1)$.

This completes the proof of \Cref{prop:sl3full:case1}. 
\end{proof}

\subsection{Case \ref{sl3full:case2}: $\Gamma$ contains a standard $\SL(2,\Z)$.} The only remaining case to be proved is Case \ref{sl3full:case2} of Theorem \ref{sl3full}. This case follows from the following proposition, whose proof is similar to the previous ones. Therefore, we only outline the proof and leave the details to the reader.

\begin{proposition}\label{prop:reducible} Let $\Gamma \subset \SL(3,\R)$ be a discrete Zariski-dense subgroup. Consider the inclusion $\SL(2,\R) \to \SL(3,\R)$ induced by an inclusion $\R^2 \hookrightarrow \R^3$ as a subspace. 
Suppose that $\Gamma$ intersects $\SL(2,\R)$ in a lattice $\Delta \subset \SL(2,\R)$.
Then $\Gamma$ acts minimally on $\Flag(\R^3)$.
\end{proposition}

\begin{proof} This is again a consequence of Proposition \ref{samicriterion}. The limit set $\pi(\Lambda_{\Gamma})\subset \PP(\R^3)$ must contain a projective line (the limit set of $\pi(\Delta)$), and similarly, one shows that $\pi^{*}(\Lambda_{\Gamma})$ must contain a projective line and Proposition \ref{samicriterion} applies.


After this, one may use the action of $\G$ to conclude easily that $\Lambda_\G = \Flag(\R^3)$.
\end{proof}

The analog of \Cref{prop:reducible} also holds if one replaces the reducible embedding of $\SL(2,\R)$ into $\SL(3,\R)$ by the irreducible one. (We would like to thank Sami Douba for pointing this out and explaining the proof.)

\begin{proposition}
    Any discrete Zariski-dense subgroup $\Gamma \subset \SL(3,\R)$, intersecting $\operatorname{SO}(2,1)$ in a lattice, must act minimally on $\Flag(\R^3)$.
\end{proposition}

\begin{proof}[Proof sketch]
$\operatorname{SO}(2,1)$ preserves a closed disk $\mathbb{D} \subset \mathbb{P}(\mathbb{R}^3)$. The limit set in $\Flag(\mathbb{R}^3)$ of any lattice in $\operatorname{SO}(2,1)$ is
\[
\Sigma \coloneqq \{ (p,l) :\ p \in \partial \mathbb{D},\, l \text{ tangent to } \partial \mathbb{D} \text{ at } p \}.
\]
Clearly, $\Sigma \subset \Lambda_\Gamma$.

Using Zariski density of $\Gamma$, one finds $\gamma \in \Gamma$ such that $\mathbb{D}$ and $\gamma \mathbb{D}$ intersect but neither contains the other. For a suitable hyperbolic element $h \in \gamma \Delta \gamma^{-1}$, where $\Delta \coloneqq \Gamma \cap \operatorname{SO}(2,1)$, with fixed points $h^\pm$ in the arc $\partial (\gamma \mathbb{D}) \smallsetminus \mathbb{D}$, the accumulation set of $h^n(\Sigma)$ as $n\to\infty$ contains the fiber $\pi^{-1}(h^+)$, which must lie in $\Lambda_\Gamma$. 
Similarly, $\Lambda_\Gamma$ also contains a full fiber of the dual projection $\pi^*: \Flag(\mathbb{R}^3) \to \mathbb{P}((\mathbb{R}^3)^*)$. 
The conclusion then follows from \Cref{samicriterion}.
\end{proof}

\subsection{Final remarks}
 Let $a$ be a diagonal matrix in $G = \SL(3,\R)$ with positive diagonal entries and precisely two equal eigenvalues.
 Let $H= \< a,N_a\>\subset G$, where $N_a$ is the stable horocyclic subgroup corresponding to $a$. 
 One can identify $G/H$ with the homogeneous space $$S = \{(F_1,F_2, F_3) :\ F_i = (l_i,V_i) \in \Flag(\R^3) \text{ are pairwise distinct such that } l_1 = l_2 = l_3 \},$$ i.e., the set of all pairwise distinct triples of complete flags in $\R^3$ sharing the same one-dimensional subspace. Using similar arguments as above, one can prove the following:

\begin{proposition}\label{sl3full2} Suppose that a subgroup $\Gamma < \SL(3,\Z)$, which is Zariski-dense in $G = \SL(3,\R)$, satisfies one of the following:
\begin{enumerate}
\item $\Gamma$ contains a subgroup isomorphic to $\Z^2$.
\item $\Gamma$ contains an infinite order element with at least one complex eigenvalue.
\end{enumerate}
Then, $\Gamma$ acts minimally on $G/H$. (Equivalently, $H$ acts minimally on $\Gamma\backslash G$). 
\end{proposition}

Finally, as evidence for a positive answer to \Cref{ques:minimal_lattice} in the case $G = \operatorname{SL}(3,\mathbb{R})$, we conclude this section with the following result, which shows that limits of Anosov surface groups in $G$ never act minimally on $\Flag(\mathbb{R}^3)$, indicating that the behavior of limits of Anosov surface groups in $\operatorname{SL}(3,\mathbb{R})$ differs significantly from the ones in $\operatorname{PSL}(2,\mathbb{C})$.

\begin{proposition}
 Let $S$ be a closed orientable surface of genus at least $2$ and let $\rho_n : \pi_1(S) \to \PSL(3,\R)$ be a sequence of Anosov representations algebraically converging to a representation $\rho_\infty : \pi_1(S) \to \PSL(3,\R)$. 
 Then $\rho_\infty(\pi_1(S))$ does not act minimally on $\Flag(\R^3)$.
\end{proposition}

\begin{proof}
It is known that $\rho_\infty$ is discrete and faithful. 
If $\rho_\infty$ is not Zariski-dense, then its image lies in a proper algebraic subgroup of $\PSL(3,\R)$ and hence cannot act minimally on $\Flag(\R^3)$.
So, we can assume that $\rho_\infty$ is Zariski-dense. Then, there exists $\gamma_0\in\pi_1(S)$ such that $\rho_\infty(\gamma_0)$ is loxodromic (by \cite{Benoist}). Let $x^\pm\in \Flag(\R^3)$ be the attractive and repulsive fixed points of $\rho_\infty(\gamma_0)$. Since $\rho_n(\gamma_0)\to \rho_\infty(\gamma_0)$ as $n\to\infty$, there exists a sequence $b_n$ in $\PSL(3,\R)$ such that $b_n\to 1$ and for all $n\in\N$, the attractive and repulsive fixed points of $b_n\rho_n(\gamma_0)b_n^{-1}$ are $x^\pm$. Let $\rho'_n: \pi_1(S) \to \PSL(3,\R)$ be given by $\rho'_n(\gamma) = b_n \rho_n(\gamma) b_n^{-1}$, $\gamma\in\Gamma$. Since $b_n\to 1$, it follows that $\rho'_n\to \rho_\infty$ as $n\to\infty$. Also, since $\rho'_n(\pi_1(S))$ is Anosov, its limit set $\Lambda_{\rho'_n(\pi_1(S))}$ is a topological circle in $\Flag(\R^3)$, which passes through $x^\pm$.

The set of points $\mathcal{C}(\{x^-,x^+\})$ in $\Flag(\R^3)$ decomposes into six connected components (see \cite{SSV}). Since $\Lambda_{\rho'_n(\pi_1(S))} \smallsetminus \{x^-,x^+\}$ has only two connected components and $\Lambda_{\rho'_n(\pi_1(S))} \smallsetminus \{x^-,x^+\}\subset \mathcal{C}(\{x^-,x^+\})$, after passing to a subsequence, we can assume that there exists a connected component $\Omega$ of $\mathcal{C}({x^-,x^+})$ not intersected by $\Lambda_{\rho'_n(\pi_1(S))}$ for all $n$.

To finish the proof, we show that $\Omega\cap \Lambda_{\rho_\infty(\pi_1(S))} = \emptyset$. Indeed, the set of attractive points of loxodromic elements of $\rho_\infty(\pi_1(S))$ is dense in $\Lambda_{\rho_\infty(\pi_1(S))}$ (by \cite{Benoist}). If, to the contrary, $\Omega\cap \Lambda_{\rho_\infty(\pi_1(S))} \ne \emptyset$, then there exists $\gamma_1\in\pi_1(S)$ such that $\rho_\infty(\gamma_1)$ is loxodromic with attractive fixed point $x_1^+$ lying in $\Omega$. Since $\rho_n(\gamma_1)\to \rho_\infty(\gamma_1)$ as $n\to\infty$, it follows that the attractive fixed points of $\rho_n(\gamma_1)$ eventually enter $\Omega$, a contradiction.
\end{proof}

\appendix
\section{\texorpdfstring{Zariski-dense discrete subgroups of $\SL(n,\R)$ with some Jordan projections lying on the boundary of the limit cone}{}}

In this appendix, $G$ denotes the group $\SL(n,\R)$, $n\ge 3$. 
For $g\in G$, let $\mu_1(g)\ge\dots\ge\mu_n(g)$ (resp. $\lambda_1(g)\ge \dots\ge \lambda_n(g)$) denote the the logarithms of the singular values (resp. of the moduli of the eigenvalues) of $g$.
Using the usual identification of $\mathfrak{a}$ with $\{ (x_1,\dots,x_n) :\ \sum_i x_i = 0\} \subset \R^n$, $\mu(g) = (\mu_1(g), \dots, \mu_n(g)) \in \mathfrak{a}^+$ is the Cartan projection of $g$ (cf. \S\ref{sec:notation}). 
Moreover, the Jordan projection $\lambda: G \to \mathfrak{a}^+$ is defined by $\lambda(g) = (\lambda_1(g), \dots, \lambda_n(g))$.

If $\Delta$ is a Zariski-dense subsemigroup of $G$, then
the set of loxodromic\footnote{An element $g\in \SL(n,\R)$ is called {\em loxodromic} if moduli of the eigenvalues of $g$ are pairwise distinct or, equivalently, $\lambda(g)\in\operatorname{int}\mathfrak{a}^+$.} elements of $\Delta$ is also Zariski-dense in $G$ (see \cite[Prop. 6.11]{Benoist-Quint-Book}).
Benoist \cite{Benoist} associated with $\Delta$ a closed cone $\LL_\Delta \subset \mathfrak{a}^+$, called the \textit{limit cone} of $\Delta$, which is defined as the smallest closed cone in $\mathfrak{a}^+$ containing the Jordan projections of the loxodromic elements of $\Delta$.
He showed that the limit cone $\LL_\Delta$ is convex and has a nonempty interior.
The limit cone captures many important dynamical information concerning the action of $\Delta$ on various homogeneous spaces of $G$.
However, the boundary of the limit cone appears to be a bit mysterious object.

\medskip

The goal of this appendix to prove the following result:

\begin{theorem}\label{thm:lc_bd}
    Let $\Delta$ be a Zariski-dense subgroup of $G = \SL(n,\R)$, where $n\ge 3$. Then, there exist loxodromic elements $a,b\in \Delta$ such that the following hold:
    \begin{enumerate}
        \item $\G\coloneqq \< a,b\>$ is a Zariski-dense Anosov subgroup of $G$.
        \item The Jordan projection of $a$ lies on the boundary of the limit cone $\LL_{\G}$.
    \end{enumerate}
\end{theorem}

We briefly discuss some importance of studying the boundary of the limit cone. Let $P = MAN$ be a Langlands decomposition of the minimal parabolic subgroup $P$.
If $\G$ is a Zariski-dense discrete subgroup of $G$, then the set $E_\G = \{[g]\in \Gamma\backslash G :\ gP\in \Lambda_\G\}$ is the unique $P$-minimal subset of $\Gamma\backslash G$. 
For example, if $\G$ has a full limit set in $G/P$ (such as when $\G$ is a lattice or a discrete group as described in \Cref{thm:main}), then $E_\G = \Gamma\backslash G$.
It is an intriguing prospect to study the $NM$-orbit closures;
for instance,  when does $NM$ act minimally  on $E_\G$?
Landesberg-Oh \cite{LANDESBERG_OH_2024} showed a connection between this question and the investigation of \textit{horospherical} limit points of $\G$.
They also proved that if there exists a loxodromic element $\g\in \G$ whose Jordan projection lies on the boundary of $\LL_\G$, then the $NM$-action on $E_\G$ is non-minimal.
Using Thurston's work \cite{thurston1998minimal}, they identified numerous such subgroups $\G$ in $\operatorname{PSL}(2,\R)\times \operatorname{PSL}(2,\R)$. 
In this regard, \Cref{thm:lc_bd} provides many examples of such subgroups in $\SL(n,\R)$, $n\ge 3$.

We now give a proof of \Cref{thm:lc_bd}. 
Throughout the proof, we adopt the following notation: For $g\in G$ and $1\le i,j\le n$,
\[
 \mu_{i,j}(g) \coloneqq \mu_i(g) - \mu_j(g), 
 \quad \lambda_{i,j}(g) \coloneqq \lambda_i(g) - \lambda_j(g).
\]
The main ingredient in the proof is \Cref{lem:wl} and part of our argument follows \cite[Sect. 4.3]{Benoist}.

\begin{proof}[Proof of \Cref{thm:lc_bd}]
    Since the limit cone has a nonempty interior, there exists a loxodromic element $a_1 \in \Delta$ such that $\lambda_{1,2}(a_1) \neq \lambda_{n-1,n}(a_1)$. We replace $a_1$ with $a_1^2$, if needed, to further assume that all eigenvalues of $a_1$ are positive. 
    Since conjugations by elements of $G$ does not affect the Jordan projections, for convenience, after conjugating $\Delta$ by some $g \in G$, we can (and will) assume that $a_1 \in A^+$. In this case, $\lambda_{j,j+1}(a^l_1) = \mu_{j,j+1}(a^l_1)$ for all $j=1,\dots,n-1$ and $l\in\Z$.
    
    By \cite[Prop. 4.4]{Tits_free}, for all integers $l\in\N$, the union of all Zariski-closed and Zariski-connected proper subgroups of $G$ containing $a^l_1$ is contained in a proper Zariski-closed subset $F\subset G$.
    Let $b_1\in \Delta$ be another loxodromic element (with positive eigenvalues) such that:
    \begin{enumerate}[label=(\alph*)]
        \item\label{cdn1:thm:lc_bd} 
        $b_1\not\in F$,

        \item\label{cdn2:thm:lc_bd} 
        $\max\Bigl\{\frac{\lambda_{1,2}(b_1^{\pm1})}{\lambda_{n-1,n}(b_1^{\pm1})}\Bigr\} 
        <
        \max\Bigl\{\frac{\lambda_{1,2}(a_1^{\pm1})}{\lambda_{n-1,n}(a_1^{\pm1})}\Bigr\}$, and
        \item\label{cdn3:thm:lc_bd} 
        if $x^\pm, y^\pm \in G/P$ are the attractive/repulsive fixed points of $a_1,b_1$, respectively, then $x^\pm, y^\pm$ are pairwise in general position.
    \end{enumerate}
    
    Condition \ref{cdn1:thm:lc_bd} ensures that for all $l,k\in\N$, the subgroup $\< a_1^l,b_1^k\>$ is a Zariski-dense subgroup of $\SL(n,\R)$. 
    We can (and will) choose sufficiently large $k$ and $l$ so that $\<a_1^l,b_1^k\>$ is an Anosov subgroup of $G$, which is naturally isomorphic to the free group on two generators; see \cite{DK23}. Pick even larger $k$ and $l$, if needed, so that we can\footnote{The only essential fact needed in the proof of this lemma is the existence of closed neighborhoods $D_{x}$ and $D_{y}$ of $\{x^\pm\}$ and $\{y^\pm\}$, respectively, which are in general position to each other and where the cyclic groups $\<a_1^l\>$ and $\<b_1^k\>$ can play ping-pong. At the expense of choosing larger values of $k$ and $l$, such neighborhoods can be constructed easily due to condition \ref{cdn3:thm:lc_bd} above; see, for example, \cite[Rem. 6.4]{DK23}.} 
    apply \Cref{lem:wl} to obtain the following: 
    If 
    $a \coloneqq a_1^l$ and $b\coloneqq b_1^k$,
    then there exists $C>0$ such that if $w = a^{p_r}b^{q_r}\cdots a^{p_1}b^{q_1}$ is any reduced word in the letters $a$ and $b$, then for all $j=1,\dots,n-1$,
\begingroup
\makeatletter\def\f@size{10}\check@mathfonts
    \begin{align}\label{eqn:apn:1}
    \begin{split}
        \sum_{i=1}^r\mu_{j,j+1}(a^{p_i})
        + \sum_{i=1}^r\mu_{j,j+1}(b^{q_i})
        - Cr
        \LE
        \mu_{j,j+1}(w)
        \LE
        \sum_{i=1}^r\mu_{j,j+1}(a^{p_i})
        + \sum_{i=1}^r\mu_{j,j+1}(b^{q_i})
        + Cr.
        \end{split}
    \end{align}
\endgroup
    Let us further assume that $p_r, q_1 \ne 0$ (in particular, $w$ is cyclically reduced).
    Then applying \eqref{eqn:apn:1} to $w^m$ and using the fact $\lambda_{j,j+1}(w) = \lim_{m\to\infty}(\mu_{j,j+1}(w^m))/m$, we obtain
\begingroup
\makeatletter\def\f@size{10}\check@mathfonts
    \begin{align}\label{eqn:apn:2}
    \begin{split}
        \sum_{i=1}^r\lambda_{j,j+1}(a^{p_i})
        + \sum_{i=1}^r\left(\mu_{j,j+1}(b^{q_i})
        - C\right)
        \LE
        \lambda_{j,j+1}(w)
        \LE
        \sum_{i=1}^r\lambda_{j,j+1}(a^{p_i})
        + \sum_{i=1}^r\left(\mu_{j,j+1}(b^{q_i})
        + C\right).
        \end{split}
    \end{align}
\endgroup
    Replacing $b$ by a suitably large power of it, we can (and will) assume that $\left(\mu_{j,j+1}(b^{q_i})- C\right) > 0$ for all $j=1,\dots,n-1$ and $q_i \in\Z\smallsetminus\{0\}$.
    Therefore, 
    \begin{align}\label{eqn:apn:3}
        \begin{split}
        \frac{\lambda_{1,2}(w)}{\lambda_{n-1,n}(w)}
        &\LE 
        \frac{\sum_{i=1}^r\lambda_{1,2}(a^{p_i})
        + \sum_{i=1}^r\left(\mu_{1,2}(b^{q_i})
        + C\right)}{\sum_{i=1}^r\lambda_{n-1,n}(a^{p_i})
        + \sum_{i=1}^r\left(\mu_{n-1,n}(b^{q_i})
        - C\right)} \\
        &\LE \max\biggl\{ \frac{\lambda_{1,2}(a^{p_i})}{\lambda_{n-1,n}(a^{p_i})}, \,  \frac{\mu_{1,2}(b^{q_i})
        + C}{\mu_{n-1,n}(b^{q_i})
        - C} :\ i=1,\dots, r\biggr\}\\
        &\LE \max \biggl\{ \biggl(\frac{\lambda_{1,2}(a)}{\lambda_{n-1,n}(a)}\biggr)^{\pm 1}
        ,\, \frac{\mu_{1,2}(b^{q_i})
        + C}{\mu_{n-1,n}(b^{q_i})
        - C} :\ i=1,\dots, r\biggr\}.
        \end{split}
    \end{align}
    In the second inequality above, we used the following fact: If $c_1,\dots,c_k,d_1,\dots,d_k$ are positive real numbers, then 
    $\frac{c_1+\dots+c_k}{d_1+\dots+d_k} \le \max\bigl\{\frac{c_1}{d_1},\dots,\frac{c_k}{d_k} \bigr\}$.
    Since
    \[
     \frac{\mu_{1,2}(b^{\pm m})
        + C}{\mu_{n-1,n}(b^{\pm m})
        - C} \to 
        \frac{\lambda_{1,2}(b^{\pm1})}{\lambda_{n-1,n}(b^{\pm1})}
        \quad\text{as } m\to\infty.
    \]
    after replacing $b$ (again) by a large enough power of it, we can (and will) assume that 
    \[
     \frac{\mu_{1,2}(b^{\pm m})
        + C}{\mu_{n-1,n}(b^{\pm m})
        - C} 
        \hspace{0.5em} < \hspace{0.5em}
        \max\biggl\{\frac{\lambda_{1,2}(a^{\pm1})}{\lambda_{n-1,n}(a^{\pm 1})}\biggr\}
    \]
    for all $m\in \N$; cf. condition \ref{cdn2:thm:lc_bd}.
    
    Therefore, \eqref{eqn:apn:3} implies that in the cases when $w =a^{p_r}b^{q_r}\cdots a^{p_1}b^{q_1}$ is a reduced word with $p_r,q_1\ne 0$, $w \in \<a\>$, or $w\in \<b\>$ (cf. condition \ref{cdn2:thm:lc_bd}), then
    \begin{equation}\label{eqn:apn:4}
        \frac{\lambda_{1,2}(w)}{\lambda_{n-1,n}(w)} 
        \LE 
        \max\biggl\{\frac{\lambda_{1,2}(a^{\pm1})}{\lambda_{n-1,n}(a^{\pm 1})}\biggr\}.
    \end{equation}
    Since every word in the letters $a$ and $b$ is conjugate to one of the above possibilities, and eigenvalues of a matrix are conjugation invariant, we see that the inequality \eqref{eqn:apn:4} holds for all words $w$ in letters $a$ and $b$.
    Thus, the Jordan projections of $a$ and $a^{-1}$ lie on boundary of the limit cone $\LL_{\<a,b\>}$.
\end{proof}

\end{document}